\documentclass{amsart}

\usepackage[latin1]{inputenc}
\usepackage[english]{babel}
\usepackage{amstext}
\usepackage{latexsym}
\usepackage{amsmath}
\usepackage{amssymb}
\usepackage{amsfonts}
\usepackage{graphicx}
\usepackage{amssymb}
\usepackage{amsthm}
\usepackage{latexsym}
\usepackage{epsfig}
\usepackage{tikz}
\usetikzlibrary{matrix,arrows,decorations.pathmorphing}
\usepackage{graphicx}
\usepackage{enumitem}
\DeclareRobustCommand{\SkipTocEntry}[4]{}

\input xy
\xyoption{all}
\begin{document}

\newtheorem{teo}{Theorem}[section]
\newtheorem{cor}[teo]{Corollary}
\newtheorem{lema}[teo]{Lemma}
\newtheorem{prop}[teo]{Proposition}
\newtheorem{nota}{Notation}
\newtheorem{defn}[teo]{Definition}
\newtheorem*{defn*}{Definition}
\newtheorem*{teosim}{Theorem A}
\newtheorem*{teocon}{Theorem B}
\newtheorem*{coro}{Corollary}
\newtheorem*{inv}{Invariance Principle}
\newtheorem*{af}{Claim}
\newtheorem{obs}[teo]{Remark}
\newtheorem{corol}{Corollary}

\title[$C^r$-density of (non-uniform) hyperbolicity in $PH^r_{\omega}(M)$]{$C^r$-density of (non-uniform) hyperbolicity in partially hyperbolic symplectic diffeomorphisms}
\author{Karina Marin}
\address{IMPA- Estrada D. Castorina 110, Jardim Bot\^anico, 22460-320 Rio de Janeiro -Brazil.}
\email{kmarin@impa.br}

\begin{abstract}
We use the Invariance Principle of Avila and Viana to prove that every partially hyperbolic symplectic diffeomorphism with 2-dimensional center bundle, having a periodic point and satisfying certain pinching and bunching conditions, can be $C^r$-approximated by non-uniformly hyperbolic diffeomorphisms.
\end{abstract}

	
\maketitle
\tableofcontents

\section{Introduction}

In the theory of Dynamical Systems, hyperbolicity is a core concept whose roots may be traced back to Hadamard and Perron and which was first formalized by Smale \cite{Sm} in the 1960's. It implies several features that are most effective to describe the system's dynamical behavior.

While Smale's uniform hyperbolicity was soon realized to be a fairly restrictive property, a more flexible version was proposed by Pesin \cite{P} about a decade later: one speaks of \emph{non-uniform hyperbolicity} when all the Lyapunov exponents are different from zero almost everywhere with respect to some preferred invariant measure (for instance, a volume measure).

While being more general, non-uniform hyperbolicity still has many important consequences, most notably: the stable manifold theorem (Pesin \cite{P}), the abundance of periodic points and Smale horseshoes (Katok \cite{K}) and the fact that the fractal dimension of invariant measures is well defined (Barreira, Pesin and Schmelling \cite{BPS}). Thus, the question of how general non-uniform hyperbolicity is, naturally arises, and indeed, it goes back to Pesin's original work.   

However, the set of non-uniformly hyperbolic systems is usually \emph{not} dense. Herman (see the presentation of Yoccoz \cite{Y}) constructed open subsets of $C^r$, with large $r$, volume-preserving diffeomorphisms admitting invariant subsets with positive volume consisting of
codimension-1 quasi-periodic tori: on such subsets all the Lyapunov exponents vanish identically. Other examples with a similar flavor were found by Cheng and Sun \cite{CS} and Xia \cite{X}.

Before that, in the early 1980's, Ma\~n\'e \cite{Ma} observed that every area-preserving diffeomorphism that is not Anosov can be $C^1$-approximated by diffeomorphisms with zero Lyapunov exponents. His arguments were completed by Bochi \cite{B1} and were extended to arbitrary dimension by Bochi and Viana \cite{B2,BV2}. In particular, Bochi \cite{B2} proved that every partially hyperbolic symplectic diffeomorphism can be $C^1$-approximated by partially hyperbolic diffeomorphisms whose center Lyapunov exponents vanish.
 
By the end of last century, Alves, Bonatti and Viana were studying the ergodic properties of partially hyperbolic diffeomorphisms. In \cite{ABV,BV3} they proved that under some amount of hyperbolicity along the center bundle (`mostly contracting' or `mostly expanding' center direction) the diffeomorphism admits finitely many physical measures and the union of their basins contains almost every point in the manifold.

This again raised the question of how frequent non-uniform hyperbolicity is, this time focusing on the partially hyperbolic setting. \emph{Can one always approximate the diffeomorphism by another whose center Lyapunov exponents are non-zero?} This question was the origin of a whole research program, focusing first on linear cocycles and dealing more recently also with non-linear systems. We refer the reader to the book of Viana \cite{V} for a detailed survey of some of the progress attained so far. 

Our own results are based on methods that were developed in these 15 years or so and may be viewed as the fulfillment of that program in the context of symplectic diffeomorphisms with $2$-dimensional center. We proved (all the keywords will be recalled in the next section):
\begin{teosim}
Let $f:M\longrightarrow M$ be a partially hyperbolic symplectic $C^r$ diffeomorphism on a compact manifold $M$. Assume that $f$ is accessible, center-bunched and pinched, the set of periodic points is non-empty, and the center bundle $E^c$ is $2$-dimensional. Then, $f$ can be $C^r$-approximated by non-uniformly hyperbolic symplectic diffeomorphisms.
\end{teosim}
Let us stress that our perturbation holds in the $C^r$ topology, for any $r\in[2,+\infty)$. The case $r=1$ is very special and much better understood. 

The first result along these lines was due to Shub and Wilkinson \cite{SW}, who proved that certain partially hyperbolic skew-products with circle center leaves can be perturbed to make the center Lyapunov exponent different from zero. Their approach relates the issue of non-uniform hyperbolicity into the analysis of the center foliation and its measure-theoretical properties, a connection that has been much deepened and clarified in the recent work of Avila, Viana and Wilkinson \cite{AVW}. 

Baraviera and Bonatti \cite{BB} extended the approach of Shub and Wilkinson to prove that any stably ergodic partially hyperbolic  diffeomorphism can be $C^1$-approximated by another for which the \emph{sum} of the center Lyapunov exponents is non-zero. In particular, this implies that every partially hyperbolic diffeomorphism with 1-dimensional center bundle can be $C^1$-approximated by non-uniformly hyperbolic systems. 

The results in \cite{BB}, together with the results mentioned above of Bochi and Viana \cite{BV2}, were used by Bochi, Fayad and Pujals \cite {BFP}, to prove that every $C^{1+\alpha}$ stably ergodic diffeomorphism can be $C^1$-approximated by non-uniformly hyperbolic ones. More recently, Avila, Crovisier and Wilkinson \cite{ACW} proved a general theorem that implies that every partially hyperbolic volume-preserving diffeomorphism can be $C^1$-approximated by non-uniformly hyperbolic systems, thus solving the question completely in the $C^1$ case.

Perturbative results in the $C^r$ topology, $r>1$, are notoriously more difficult and, in fact, there is good evidence suggesting that the conclusions may also be very different. In this regard, we refer the reader to the discussions in Chapter~12 of \cite{BDV}, Chapter~10 of \cite{V}, Theorem A of \cite{AV1} and the Conjectures in Section 6 of \cite{Pe}.

The obvious first case to look at is when the center bundle is 1-dimensional. While there are some situations where we know how to get rid of a zero center Lyapunov exponent, for example \cite{DP, SW}, $C^r$-density of hyperbolicity is not known even in this case. Our methods in the present paper, which are based on the projectivization of the center bundle, do not seem to be useful in this context. 

An important tool in our approach is the Invariance Principle, which was first developed by Furstenberg \cite{F} and Ledrappier \cite{L} for random matrices and was extended by Bonatti, G\'omez-Mont, Viana \cite{BGV} to linear cocycles over hyperbolic systems and by Avila, Viana \cite{AV1} and Avila, Santamaria, Viana \cite{ASV} to general (diffeomorphisms) cocycles. In \cite{AV1} the base dynamics is still assumed to be hyperbolic, whereas in \cite{ASV}, it is taken to be partially hyperbolic and volume-preserving.

The Invariance Principle asserts that for the Lyapunov exponents to vanish the system must exhibit rather rigid (holonomy invariant) features. Often, one can successfully exploit those features to describe the system in a rather explicit way. One fine example is the main result of Avila, Viana and Wilkinson \cite{AVW}: small perturbations of the time-$1$ map of the geodesic flow on a surface with negative curvature either are non-uniformly hyperbolic or embed into a smooth flow. 

Another fine application was made by Avila and Viana \cite{AV1}, who exhibited partially hyperbolic diffeomorphisms for which the Lyapunov exponents can not vanish because structure arising from the Invariance Principle, namely invariant line fields, is incompatible with the topology of the center leaves (which are assumed to be surfaces of genus $g>1$).

Perhaps the main novelty in this work is that we are able to use the Invariance Principle in a perturbative way, to prove that the Lyapunov exponents can be made non-zero. At a more technical level, another main novelty resides in our handling of the accessibility property, namely the way $su$-paths and their holonomies vary, under perturbations of the diffeomorphism; see Section 4 and 5. 

As an example of the reach of Theorem A, let us state the following result that is related to Question 1b) in \cite{SW}. Let $f:\mathbb{T}^{2d} \longrightarrow \mathbb{T}^{2d}$ be a Anosov symplectic $C^r$ diffeomorphism and $g_{\lambda}:\mathbb{T}^2 \longrightarrow \mathbb{T}^2 $ denote the \textit{standard map} on the $2$-torus. 
\begin{coro} If $\lambda$ is close enough to zero, then $f\times g_{\lambda}$ can be $C^r$-approximated by non-uniformly hyperbolic symplectic diffeomorphisms.
\end{coro}
\addtocontents{toc}{\SkipTocEntry}
\subsection*{Acknowledgments} 
The author would like to thank Marcelo Viana for the guidance and encouragement during her Ph.D. Thesis at IMPA which originates this work, Artur Avila for useful conversations and the anonymous referee for a thorough revision of the paper that greatly helped improve the presentation. The author has been partially supported by IMPA, CNPq, CAPES and FAPERJ. 

\section{Preliminaries and Statements}

From now on, $M$ will denote a compact manifold and $f:M\longrightarrow M$ a partially hyperbolic diffeomorphism. In this section we define this and other related notions. For more information see \cite{BDV,HPS,Sh}. 

A diffeomorphism $f:M\longrightarrow M$ of a compact manifold $M$ is \textit{partially hyperbolic} if there exist a nontrivial splitting of the tangent bundle $$TM=E^{s}\oplus E^{c}\oplus E^{u}$$ invariant under the derivative map $Df$, a Riemannian metric $\left\| \cdot \right\|$ on $M$, and positive continuous functions $\chi$, $\widehat{\chi}$, $\nu$, $\widehat{\nu}$, $\gamma$, $\widehat{\gamma}$ with 
$$\chi< \nu < 1 < \widehat{\nu}^{-1} < \widehat{\chi}^{-1} \quad  \text{and} \quad \nu< \gamma < \widehat{\gamma}^{-1}< \widehat{\nu}^{-1},$$ such that for any unit vector $v\in T_{p}M$, 
\begin{equation}\label{ph}
\begin{aligned}
\chi(p)< &\left\| Df_{p}(v) \right\|< \nu(p) \quad \quad \text{if} \; v\in E^{s}(p), \\
\gamma(p) < &\left\| Df_{p}(v) \right\|< \widehat{\gamma}(p)^{-1} \quad \text{if} \; v\in E^{c}(p),\\
\widehat{\nu}(p)^{-1}< &\left\| Df_{p}(v) \right\| < \widehat{\chi}(p)^{-1} \quad \text{if} \; v\in E^{u}(p).
\end{aligned}
\end{equation}
Partial hyperbolicity is a $C^1$-open condition, that is, any diffeomorphism sufficiently $C^1$-close to a partially hyperbolic diffeomorphism is itself partially hyperbolic. 

Here, $\mu$ will always denote a probability measure and $\omega$ a symplectic form. The Lebesgue class of $M$ is the measure class of the volume induced by any Riemannian metric and we say that $f$ is \textit{volume-preserving} if it preserves some probability measure in this class. 

For $r\geq 2$, denote by $PH^r_{\mu}(M)$ the set of partially hyperbolic volume-preserving $C^r$ diffeomorphisms. If $M$ is a symplectic manifold, denote by $PH^{r}_{\omega}(M)$ the set of partially hyperbolic $C^{r}$ diffeomorphisms preserving $\omega$. 

For every partially hyperbolic diffeomorphism the stable and unstable bundles $E^{s}$ and $E^{u}$ are uniquely integrable and their integral manifolds form two tranverse (continuous) foliations $W^{s}$ and $W^{u}$, whose leaves are immersed submanifolds of the same class of differentiability as $f$. These foliations are called the \textit{strong-stable} and \textit{strong-unstable} foliations. They are invariant under $f$, in the sense that
$$f(W^{s}(x))=W^{s}(f(x)) \qquad \text{and}\qquad f(W^{u}(x))=W^{u}(f(x)),$$ where $W^{s}(x)$ and $W^{u}(x)$ denote the leaves of $W^{s}$ and $W^{u}$, respectively, passing through any $x\in M$. 

Given two points $x,y\in M$, $x$ is \textit{accessible} from $y$ if there exists a path that connects $x$ to $y$, which is a concatenation of finitely many subpaths, each of which lies entirely in a single leaf of $W^u$ or a single leaf of $W^s$. We call this type of path, \textit{su-path}. This define an equivalence relation and we say that $f$ is \textit{accessible} if $M$ is the unique accessibility class. 
\begin{defn}[$\alpha$-pinched]\label{holder} Let $f$ be a partially hyperbolic diffeomorphism and $\alpha>0$. We say that $f$ is $\alpha$-pinched if the functions in Equation (\ref{ph}) satisfy,  
\begin{equation*}
\begin{aligned}
\nu &< \gamma\, \chi^{\alpha} \quad \text{and} \quad \; \widehat{\nu} < \widehat{\gamma}\, \widehat{\chi}^{\alpha}, \\
\nu &< \gamma\, \widehat{\chi}^{\alpha} \quad \text{and} \quad \; \widehat{\nu} < \widehat{\gamma}\, \chi^{\alpha}. 
\end{aligned}
\end{equation*} 
\end{defn}
Notice this is a $C^1$-open property. Moreover, every partially hyperbolic diffeomorphism is $\alpha$-pinched for some $\alpha>0$. From the works in \cite{A,HPS,PSW1}, we know that the center bundle $E^c$, and the $W^s$ and $W^u$ holonomies are $\alpha$-H\"older for every $\alpha$-pinched $C^2$ diffeomorphism. 
 
We say that a partially hyperbolic diffeomorphism is \textit{center bunched} if 
$$\nu< \gamma \widehat{\gamma} \quad \text{and} \quad \widehat{\nu}< \gamma \widehat{\gamma}.$$
By Theorem 0.1 of \cite{BW}, in the volume-preserving setting, center bunching together with accessibility implies that the diffeomorphism is ergodic. 

Since we want to use the Invariance Principle, we need $f$ to be center bunched. However, this notion is not enough for our work and we need to define a stronger condition.
\begin{defn}[$\alpha$-bunched]\label{bunched} Let $f$ be a partially hyperbolic diffeomorphism and $\alpha>0$. We say that $f$ is $\alpha$-bunched if the functions in Equation (\ref{ph}) satisfy, 
$$\nu^{\alpha} < \gamma \widehat{\gamma} \qquad \text{and} \qquad \widehat{\nu}^{\alpha}< \gamma \widehat{\gamma}.$$ 
\end{defn}
This is also a $C^1$-open property. If $Df\vert E^c$ is an isometry, then the condition holds for every $\alpha>0$. Moreover, $\alpha$-bunched implies center bunched if $\alpha\leq 1$. Notice that as $\alpha$ decreases it is more difficult to have $\alpha$-bunched, contrary to what happens with the condition of $\alpha$-pinched. 

At this point we are ready to give the precise definition of the set of partially hyperbolic systems where Theorem A holds. Recall $*\in \{\mu, \omega\}$ where $\mu$ denotes some probability measure in the Lebesgue class and $\omega$ denotes a symplectic form. 
\begin{defn} If $r\geq 2$, we will denote $ B^{r}_{*}(M)$, the subset of $PH^{r}_{*}(M)$ where $f$ is accessible, $\alpha$-pinched and $\alpha$-bunched for some $\alpha>0$, and the center bundle $E^c$ is 2-dimensional.
\end{defn}
We want to remark two properties of the set $B^{r}_{*}(M)$. Avila and Viana in \cite{AV2} proved, under the hypothesis of 2-dimensional center bundle, that accessibility is a $C^1$-open property. This implies that $B^{r}_{*}(M)$ is an open set. We will give the precise statement and some ideas of the proof of this result in Section 5. Moreover, as already mentioned every $f\in B^{r}_{*}(M)$ is ergodic. 

The conditions of $\alpha$-pinched and $\alpha$-bunched will allow us to use the Invariance Principle. In Section 3, we are going to apply both of them to prove that the cocycle associated to $f$, $F=Df\vert E^c$, admits holonomies. Moreover, the $\alpha$-pinched condition is going to be use also in Section 4 where we need the $W^s$ and $W^u$ holonomies to be $\alpha$-H\"older in order to estimate how $su$-paths change when we perturb the diffeomorphism. 

If $f$ is a volume-preserving $C^1$ diffeomorphism, by Oseledets Theorem for $\mu$-almost every point $x\in M$, there exist $k(x)\in \mathbb{N}$, real numbers $\lambda_1(f,x)> \cdots > \lambda_{k(x)}(f,x)$ and a splitting $T_{x}M=E^{1}_x\oplus \cdots \oplus E^{k(x)}_x$ of the tangent bundle at $x$, all depending measurably on the point $x$, such that
$$\lim\limits_{n\rightarrow \pm \infty} \frac{1}{n} \text{log} \left\|Df^{n}_x(v)\right\|= \lambda_j(f,x) \quad \text{for all} \; v\in E^j_x \setminus \{0\}. $$
The real numbers $\lambda_j(f,x)$ are the Lyapunov exponents of $x$. We say that $f$ is \textit{non-uniformly hyperbolic} if the set of points with non-zero Lyapunov exponents has full measure.

If $f\in B^{r}_{*}(M)$, the Oseledets Theorem can be applied and because of the ergodicity, the functions $k$ and $\lambda_j$ are constants almost everywhere. Moreover, the Oseledets splitting is a measurable refinement of the original splitting and we can consider the Lyapunov exponents of $E^c$. They are called \textit{center Lyapunov exponents} and will be denoted by $\lambda_1^c(f)$ and $\lambda_2^c(f)$.

If $dim\; M=2d$ and $f$ is a symplectic ergodic diffeomorphism, then $$\lambda_j(f)=-\lambda_{2d-j+1}(f) \quad \text{for all} \: 1\leq j\leq d.$$ Therefore, in the symplectic case $\lambda^c_1(f)=\lambda^c_2(f)$ is equivalent to $\lambda^c_1(f)=\lambda^c_2(f)=0.$ This symmetry property has been proved in \cite{BV1}.

Now we give the statement of the main result. 
\begin{teosim} Let $f\in B^{r}_{\omega}(M)$ and assume the set of periodic points of $f$ is non-empty, then $f$ can be $C^r$-approximated by non-uniformly hyperbolic symplectic diffeomorphisms.
\end{teosim}
\begin{obs} Observe that the hypothesis of existence of a periodic point in Theorem A can be replace with the hypothesis of $f$ having a periodic compact $C^r$ center leaf. In this case, we can find a symplectic diffeomorphism arbitrarily $C^r$-close to $f$ and having a periodic point. See \cite{XZ}.
\end{obs}
The proof of Theorem A relies in two principal cases determined by the periodic point being hyperbolic or elliptic. The hyperbolic case has a generalization to the volume-preserving setting with the appropriate modifications in the hypotheses.  
\begin{defn}\label{pinch} Let $f$ be a partially hyperbolic diffeomorphism and $p$ a periodic point with $n_p=per(p)$. We say that $p$ is a \emph{pinching periodic point} if $Df^{n_p}|E^c(p)$ has two real eigenvalues with different norms.
\end{defn}
Recall $*\in \{\mu, \omega\}$.
\begin{teocon} Let $f\in B^{r}_{*}(M)$ and assume $f$ has a pinching periodic point, then $f$ can be $C^r$-approximated by volume-preserving (symplectic) diffeomorphisms whose center Lyapunov exponents are different almost everywhere. In particular, $f$ can be $C^r$-approximated by diffeomorphisms with some center Lyapunov exponent non-zero.
\end{teocon}
One of the main tools in our proof is the Invariance Principle. Below we give some preliminaries and state it in the form of \cite{ASV}. 

\subsection{Invariance Principle} 
Let $f$ be a partially hyperbolic diffeomorphism and $\pi: \mathcal{V}\longrightarrow M$ a continuous vector bundle with fiber $N=\mathbb{R}^{k}$ for some $k\geq 2$. A \textit{linear cocycle} over $f:M\longrightarrow M$ is a continuous transformation, $F: \mathcal{V}\longrightarrow \mathcal{V}$, satisfying $\pi \circ F= f\circ \pi$ and acting by linear isomorphisms, $F_{x}: \mathcal{V}_{x}\longrightarrow \mathcal{V}_{f(x)}$, on the fibers. By Fustenberg, Kesten \cite{FK}, the extremal Lyapunov exponents $$\lambda_{+}(F,x)=\lim\limits_{n\rightarrow\infty} \frac{1}{n} \text{log} \left\|F_{x}^{n}\right\| \quad \text{and} \quad \lambda_{-}(F,x)=\lim\limits_{n\rightarrow\infty} \frac{1}{n} \text{log} \left\|(F_{x}^{n})^{-1}\right\|^{-1},$$ exist at $\upsilon$-almost every $x\in M$, relative to any $f$-invariant probability measure $\upsilon$. If $(f,\upsilon)$ is ergodic, then they are constant on a full $\upsilon$-measure set. It is clear that $\lambda_{-}(F,x)\leq \lambda_{+}(F,x)$ whenever they are defined.

The \textit{projective bundle} associated to a vector bundle $\pi: \mathcal{V}\longrightarrow M$ is the continuous fiber bundle $\pi: \mathbb{P}(\mathcal{V}) \longrightarrow M$ whose fibers are the projective quotients of the fibers of $\mathcal{V}$. This is a fiber bundle with smooth leaves modeled on $N=\mathbb{P}(\mathbb{R}^{k})$. 

The \textit{projective cocycle} associated to a linear cocycle $F: \mathcal{V}\longrightarrow \mathcal{V}$ is the smooth cocycle $\mathbb{P}(F): \mathbb{P}(\mathcal{V})\longrightarrow \mathbb{P}(\mathcal{V}) $ whose action $\mathbb{P}(F_x): \mathbb{P}(\mathcal{V}_{x})\longrightarrow \mathbb{P}(\mathcal{V}_{f(x)})$ on the fibers is given by the projectivization of $F_{x}$. 

For every $f$-invariant probability measure $\upsilon$, there exists a $\mathbb{P}(F)$-invariant probability measure $m$ that project down to $\upsilon$. This is true because the projective cocycle $\mathbb{P}(F)$ is continuous and the domain $\mathbb{P}(\mathcal{V})$ is compact. Moreover, the extremal Lyapunov exponents of $\mathbb{P}(F)$ exist and satisfy,
\begin{equation*}
\begin{aligned}
\lambda_{+}(\mathbb{P}(F), x,\xi)\leq \lambda_{+}(F,x) - \lambda_{-}(F,x)& \quad \text{and}\\
&\lambda_{-}(\mathbb{P}(F), x,\xi)\geq \lambda_{-}(F,x) - \lambda_{+}(F,x).
\end{aligned}
\end{equation*}
Let $R>0$ be fixed, then the \textit{local strong-stable leaf} $W^{s}_{loc}(x)$ of a point $x\in M$ is the neighborhood of radius $R$ around $x$ inside $W^{s}(x)$. The \textit{local strong-unstable leaf} is defined analogously. Since we are working in the context of \cite{ASV}, the choice of $R$ here will be the same than in Section 5 of that paper. 
\begin{defn}\label{hol} We call \emph{invariant stable holonomy} for $\mathbb{P}(F)$ a family $h^{s}$ of homeomorphisms $h^{s}_{x,y}: \mathbb{P}(\mathcal{V}_{x})\longrightarrow \mathbb{P}(\mathcal{V}_{y})$, defined for all $x$ and $y$ in the same strong-stable leaf of $f$ and satisfying
\begin{enumerate} [label=\emph{(\alph*)}]
\item $h^{s}_{y,z}\circ h^{s}_{x,y}= h^{s}_{x,z}$ and $h^{s}_{x,x}=Id$;
\item $\mathbb{P}(F_{y})\circ h^{s}_{x,y}= h^{s}_{f(x),f(y)}\circ \mathbb{P}(F_{x})$\; 
\item $(x,y, \xi)\mapsto h^{s}_{x,y}(\xi)$ is continuous when $(x,y)$ varies in the set of pairs of points in the same local strong-stable leaf;
\item there are $C>0$ and $\eta >0$ such that $h^{s}_{x,y}$ is $(C,\eta)$-H\"older continuous for every $x$ and $y$ in the same local strong-stable leaf.
\end{enumerate}
\emph{Invariant unstable holonomy} is defined analogously, for pairs of points in the same strong-unstable leaf.
\end{defn}
Let $m$ be a probability measure in $\mathbb{P}(\mathcal{V})$ and $\upsilon=\pi_{*}m$ be its projection, then there exists a disintegration of $m$ into conditional probabilities $\left\{m_{x} : x\in M \right\}$ along the fibers which is essentially unique, that is, a measurable family of probability measures such that $m_x(\mathbb{P}(\mathcal{V}_x))=1$ for almost every $x\in M$ and $$m(U)=\int m_x(U\cap \mathbb{P}(\mathcal{V}_x)) d\upsilon(x),$$ for every measurable set $U\subset \mathbb{P}(\mathcal{V})$. See \cite{Rok}.
\begin{defn} A disintegration $\left\{m_{x} : x\in M \right\}$ is \emph{s-invariant} if
$$(h^{s}_{x,y})_{*}m_{x}=m_{y} \quad \text{for every} \; x\; \text{and}\; y \; \text{in the same strong-stable leaf.}$$
The definition of \emph{u-invariant} is analogous and we say the disintegration is \emph{bi-invariant} if it is both s-invariant and u-invariant.
\end{defn}
\begin{inv} [Theorem B of \cite{ASV}]  Let $f:M\longrightarrow M$ be a $C^{2}$ partially hyperbolic, volume-preserving, center bunched diffeomorphism and $\mu$ be an invariant probability measure in the Lebesgue class. Let $F$ be a linear cocycle such that $\mathbb{P}(F)$ admits holonomies and suppose that $\lambda_{-}(F,x)=\lambda_{+}(F,x)$ at $\mu$-almost every point. 

Then, every $\mathbb{P}(F)$-invariant probability $m$ on the projective fiber bundle $\mathbb{P}(\mathcal{V})$ with $\pi_{*}m=\mu$ admits a disintegration $\left\{m_{x} : x\in M \right\}$ along the fibers such that 
\begin{enumerate} [label=\emph{(\alph*)}]
\item  the disintegration is bi-invariant over a full measure bi-saturated set $M_{F}\subset M$;
\item  if $f$ is accessible, then $M_{F}=M$ and the conditional probabilities $m_{x}$ depend continuously on the base point $x\in M$, relative to the $\text{weak}^{*}$ topology.
\end{enumerate}
\end{inv}

\subsection{Toy Model}
Given $f\in B^r_{*}(M)$, the linear cocycle $F=Df\vert E^c$ will be called \textit{center derivative cocycle} for $f$. In Section 3, we prove that we can apply the Invariance Principle to this cocycle when $\lambda^c_1(f)=\lambda^c_2(f)$ almost everywhere. For this, we prove the existence of holonomies for $\mathbb{P}(F)$ and study how they vary under the perturbation of the diffeomorphism. The main results are Proposition \ref{continuity} and Corollary \ref{compo}.

We consider the following toy model to explain the main ideas and steps for the proof of Theorem B. These ideas are classical and have already appeared, for example, in \cite{AV1,VY}. 

Suppose $f\in B^r_{*}(M)$, $p$ is a pinching fixed point (Definition \ref{pinch}), there exists $z\in M$ such that $z\in W^{ss}(p)\cap W^{uu}(p)$ and $\lambda^c_1(f)=\lambda^c_2(f)$ almost everywhere. Then, we can apply the Invariance Principle for any $\mathbb{P}(F)$-invariant probability measure $m$ with $\pi_{*}m=\mu$. Therefore, there exists a disintegration $\{m_x : x\in M\}$ such that $(h^s(f))_{*} m_p=m_z$ and $(h^u(f))_{*} m_p=m_z$ where $h^s(f)$ and $h^u(f)$ denote the holonomies along the strong-stable and strong-unstable leaves respectively. Moreover, there exist $a,b\in \mathbb{P}(E^c_p)$ such that $supp\; m_p\subset \{a,b\}$. 

We will make a perturbation supported in a neighborhood of $z$, $B_{\delta}(z)$, which has the property that $f^j(z)\notin B_{\delta}(z)$ for every $j\in \mathbb{Z}\setminus \{0\}$. Since $dim\,E^c=2$, $\mathbb{P}(F)$ is a cocycle of circle diffeomorphisms over $f$ and it makes sense to consider rotations in $\mathbb{P}(E^c_p)$. The perturbation is chosen in order to have $g$ close enough to $f$ and $h^s(g)=R_{\beta}\circ h^s(f)$ and $h^u(g)=h^u(f)$. Here, $R_{\beta}$ denotes a rotation of angle $\beta>0$. This implies that $g$ does not satisfy the Invariance Principle and therefore $\lambda^c_1(g)\neq\lambda^c_2(g)$ at almost every point.

\subsection{Strategy of the proof} 
We extend this argument to the general case when we do not necessarily have a point of homoclinic intersection. 

First we find a $su$-path from $p$ to itself with a special node $z$, which is slowly accumulated by the orbits of all the nodes including its own. This is done in Proposition \ref{pivot}. Next, we construct a sequence of $C^r$-perturbations denoted by $f_k$ and supported in $B_{\delta_k}(z)$. The details are given in Lemma \ref{pert} and Lemma \ref{pertfinal}.

In the second part of Section 4, we study how the $su$-path and the holonomies change under the variation of the diffeomorphism. The main results are Proposition \ref{su} and Proposition \ref{angulos}. In the first one, we define the continuation of the $su$-path for every $f_k$ and estimate the distance between the new nodes and the nodes of the original $su$-path. In the second one, we estimate the angle between the center bundle of $f$ and the center bundle of $f_k$. Finally, we summarized these results in Corollary \ref{holfinal}. The main observation is that the variation in the holonomies is exponentially small in $k$, although the size of the perturbations $\delta_k$ is polynomial in $k$. This will allow us to break the rigidity given by the Invariance Principle.

We are going to suppose that $\lambda^c_1(f_k)=\lambda^c_2(f_k)$ almost everywhere and apply the Invariance Principle for some probability measure $m^k$ with $\pi_{*}m^k=\mu$. This gives a family of disintegrations $\{m^k_{x} : x\in M\}$. In order to conclude the argument we need the functions $x\mapsto m^k_x$ to be equicontinuous. We are not able to prove this property, but the problem is solved using the hyperbolicity of $p$ and the results in Section 5 and 6. In Section 5, we state the theorems from \cite{AV2} and prove Proposition \ref{caminocont}. This Proposition gives some kind of continuity for $su$-paths under the variation of the diffeomorphism. In Section 6, we study the disintegration given by the Invariance Principle for some $\mathbb{P}(F)$-invariant probability measure $m$ with $\vert supp\, m_x\vert=1$. In this particular case, we obtain Proposition \ref{soporte}.

Finally, in Sections 7 we combine all these results to give the proofs of Theorem B and Theorem A. In Section 8, we apply Theorem A to partially hyperbolic diffeomorphisms of the torus. 

\section{Center Derivative Cocycle}

As already mentioned, in order to prove Theorem B we need to be able to apply the Invariance Principle to the center derivative cocycle, $F=Df\vert E^c$, when $\lambda^c_1(f)=\lambda^c_2(f)$ almost everywhere. By the definition of $B^r_{*}(M)$, we only have to prove that $\mathbb{P}(F)$ admits holonomies. Notice that the Lyapunov exponents of $F$ coincide with the center Lyapunov exponents of $f$.

For every $f\in B^r_{*}(M)$, there exists $\alpha>0$ such that $f$ is $\alpha$-pinched. Then, the center bundle $E^c$ is $\alpha$-H\"older and the center derivative cocycle $F$ is a $C^{0, \alpha}$ cocycle. Moreover, since $f$ is $\alpha$-bunched, it is enough to apply the results in Section 3 of \cite{ASV} to prove that $\mathbb{P}(F)$ admits holonomies. However, we provide a new proof that allow us to give estimations about how these holonomies change under the variation of the diffeomorphism. These are new results that we have to prove in order to be able to work in a perturbative way.  

Although the statements are for $f\in B^r_{*}(M)$, the only necessary hypotheses are the $\alpha$-pinched and $\alpha$-bunched conditions. 

Since $M$ is compact, we can define a distance in $TM$ in the following way: For every $x,y\in M$ close enough, denote $\pi_{x,y}:T_{x}M\longrightarrow T_{y}M$ the parallel transport along $\zeta$, where $\zeta$ is the geodesic satisfying $dist(x,y)=\text{length}(\zeta)$. Then, given two points $(x,v)$ and $(y,w)$ in $TM$ define $$d((x,v),(y,w))=dist(x,y) + \left\|\pi_{x,y}(v)-w\right\|.$$

To simplify the notation we are going to write $$d((x,v), (y,w))=d(v,w)\quad \text{and} \quad \pi^n_{x,y}=\pi_{f^n(x),f^n(y)}.$$

Since $f$ is $C^2$, there exists $C_0>0$ such that for every $(x,v),\, (y,w)\in TM$, $$d(Df(x,v), Df(y,w))< C_0\, d(v,w).$$

Let $V$ be a vector space with inner product and let $E_1$ and $E_2$ be subspaces of $V$. Then, define $dist(E_1, E_2)=\max \{\xi_1, \xi_2\}$ where $$\xi_1=\sup\limits_{x\in E_1, \left\|x\right\|=1}\; \; \inf\limits_{y\in E_2} \left\|x-y\right\|,$$ and $\xi_2$ is defined analogously changing the places of $E_1$ and $E_2$. If $P_{E}$ denote the orthogonal projection to the subspace $E$, then $$\inf\limits_{y\in E} \left\|x-y\right\|= \left\|x-P_{E}(x)\right\|.$$ Therefore, $\xi_1=\left\|(Id-P_{E_2})P_{E_1}\right\|$ and we have an analogous identity for $\xi_2$. 

If $x$ and $y$ are close enough, given $E_x$ and $E_y$ subspaces of $T_{x}M$ and $T_{y}M$ respectively, define $$dist(E_x, E_y)=dist(\pi_{x,y}(E_x), E_y)=dist(E_x, \pi_{y,x}(E_y)).$$ 

Since $E^c$ is $\alpha$-H\"older, there exists $C_1>0$ such that $$dist(E^c_x,E^c_y)<C_1\, dist(x,y)^{\alpha}.$$ 

Moreover, the constant $C_1$ can be taken uniform in a $C^2$ neighborhood of $f$. See for example, \cite{W}.

The next Proposition proves the existence of a family of maps for $F$, $H^s_{x,y}$, with certain properties that will imply that $\mathbb{P}(H^s_{x,y})$ defines an invariant stable holonomy for $\mathbb{P}(F)$. 
\begin{prop}\label{existence} Fix $f\in B^r_{*}(M)$ and denote $F=Df\vert E^c$. Then, for any pair of points $x,y$ in the same leaf of the strong-stable foliation $W^s$, there exist a linear homeomorphism $H^s_{x,y}:E^c_x\to E^c_y$ satisfying: 
\begin{enumerate} [label=\emph{(\alph*)}]
\item $F_{y}\circ H^{s}_{x,y}= H^{s}_{f(x),f(y)}\circ F_{x}$ and
\item $H^{s}_{y,z}\circ H^{s}_{x,y}= H^{s}_{x,z}$ and $H^{s}_{x,x}=Id$.
\end{enumerate}
\end{prop} 
\textit{Proof.} Fix $f\in B^r_{*}(M)$ and let $F=Df\vert E^c$. For $n\in \mathbb{N}$, $$F^n(x)=F(f^{n-1}(x))\circ \cdots \circ F(x),$$ and for any continuous function $\tau:M\longrightarrow \mathbb{R}^{+}$, $$\tau^n(x)=\tau(x)\tau(f(x)) \cdots \tau(f^{n-1}(x)).$$ Here, we are going to consider the continuous functions given by Equation (\ref{ph}).

If $x,y\in M$ with $y\in W^s_{loc}(x)$, then for every $n\in \mathbb{N}$ define $$A_n(x,y)=F^n(y)^{-1}\circ P_{E^c(f^n(y))} \circ \pi^n_{x,y} \circ F^n(x),$$ and $$A_0(x,y)= P_{E^c(y)} \circ \pi_{x,y}\vert E^c(x).$$

We are going to prove that this sequence is a Cauchy sequence and define $$H^s_{x,y}=\lim\limits_{n\rightarrow \infty} A_n(x,y).$$

Observe that $$A_{n+j}(x,y)=F^j(y)\circ A_n(f^j(x),f^j(y)) \circ F^j(x).$$ Therefore, once we have proved that the limit above exists, we can use this identity to demonstrate the general case when $y\in W^s(x)$ and also to conclude property (a). The proof of (b) is an easy calculation from this formula. 

In order to prove that $A_n(x,y)$ is a Cauchy sequence we find constants $C_2>0$ and $\varsigma<1$ such that $$\left\|A_{n+1}(x,y)-A_n(x,y)\right\|\leq C_2\, \nu^n(x)^{(1-\varsigma)\alpha}\,dist(x,y)^{\alpha},$$ for every $n\in \mathbb{N}\cup \{0\}$. This is a consequence of the definitions at the beginning of this Section and the following Lemma. 
\begin{lema}[Lemma 3.1, \cite{ASV}]\label{31} There exist $\widehat{C}>0$ and $\varsigma<1$ such that for all $x\in M$, $y,z \in W^s_{loc}(x)$ and $n\geq 1$, $$\prod_{j=0}^{n-1}\left\|F(f^j(y))\right\|\left\|F(f^j(z))^{-1}\right\|\leq \widehat{C}\, \nu^n(x)^{-\alpha \, \varsigma}.$$
\end{lema} 
\begin{obs}\label{bound} 
For every $y\in W^s_{loc}(x)$ we have, $$H^s_{x,y}=\sum_{j=0}^{\infty} (A_{n+1}(x,y)-A_n(x,y)) + A_0(x,y).$$
Let $C_3>0$ be a bound for $\sum_{j=0}^{\infty} \nu^n(x)^{(1-\varsigma)\alpha}$ and let $C=C_2\, C_3$, then 
\begin{equation*}
\begin{aligned}
\left\|H^s_{x,y}- A_0(x,y)\right\|&\leq \sum_{j=0}^{\infty} \left\|A_{n+1}(x,y)-A_n(x,y)\right\| \\
                                  &\leq C_2\, dist(x,y)^{\alpha}\, \sum_{j=0}^{\infty}\nu^n(x)^{(1-\varsigma)\alpha}\\
																	&\leq C\, dist(x,y)^{\alpha}.
\end{aligned} 
\end{equation*}
Then, we have proved that there exists $C>0$ such that for every $y\in W^s_{loc}(x)$, $$\left\|H^s_{x,y}\right\|\leq 1 + C\, dist(x,y)^{\alpha}.$$ Moreover, the constant $C$ depends only on $f$. 
\end{obs} 
Observe that all the estimations in the Proposition and the Remark can be taken uniform in a $C^2$ neighborhood of $f$. For every $f\in B^r_{*}(M)$, we fix this $C^2$ neighborhood and from now on, every $g$ $C^2$-close to $f$ will be understood to belong to it.

Now that we have proved the existence of $H^s_{x,y}$, we can define $h^s_{x,y}=\mathbb{P}(H^s_{x,y})$. In order to show that the family $h^s_{x,y}$ is an invariant stable holonomy for $\mathbb{P}(F)$, we need to prove property (c) in Definition \ref{hol}. 

Using the explicit formula for $H^s_{x,y}$ given in the proof of Proposition \ref{existence}, we are able to prove a stronger result, that will imply (c), but provides also an estimation about how $H^s_{x,y}$ changes under the variation of the diffeomorphism.
\begin{prop}\label{continuity} Fix $f\in B^r_{*}(M)$, $x\in M$, $y\in W^s_{loc}(x,f)$ and $a\in E^c(x,f)$. For every $\epsilon>0$ there exist $\delta>0$ and a neighborhood of $f$ in the $C^2$ topology, $\mathcal{V}(f)$, such that for every $g\in \mathcal{V}(f)$, every $w,z\in M$ with $w\in W^s_{loc}(z,g)$, $dist(x,z)<\delta$ and $dist(y,w)<\delta$ and every $b\in E^c(z,g)$ with $d(a,b)<\delta$, we have
$$d(H^s_{x,y}(f)(a), H^s_{z,w}(g)(b))< \epsilon.$$
\end{prop}
\begin{proof}
Let $F=Df\vert E^c(f)$. Similar estimations that the ones in Remark \ref{bound}, provides a $C>0$ and $\varsigma<1$ such that for any $n\geq1$, $$\left\|H^s_{x,y}(f)- A_n(f,x,y)\right\|\leq C\; \nu^n(x)^{\alpha(1-\varsigma)}$$ and $$\left\|H^s_{z,w}(g)- A_n(g,z,w)\right\|\leq C\; \nu^n(z)^{\alpha(1-\varsigma)}.$$ Then, the Proposition is a consequence of the continuity of $$A_n(f,x,y)=F^n(y)^{-1}\circ P_{E^c(f^n(y))} \circ \pi^n_{x,y} \circ F^n(x),$$ as a function of $(f,x,y)$.

More precisely, the distance $$d(A_n(f,x,y)(a), A_n(g,z,w)(b))$$ can be bounded by an expression that depends on the following terms:
$$dist(x,z), \quad dist(y,w), \quad d(a,b),$$
$$\angle (E^c(x,f), E^c(x,g)), \quad \angle (E^c(y,f), E^c(y,g)),$$ 
$$\left\|Df(f^j(x))-Dg(f^j(x))\right\|\quad \text{and} \quad \left\|Df^{-1}(f^{j+1}(y))-Dg^{-1}(f^{j+1}(y))\right\|,$$ for $j\in\{0,...,n\}.$
\end{proof}

By Proposition \ref{existence} and Proposition \ref{continuity}, the family $h^s_{x,y}=\mathbb{P}(H^s_{x,y})$ is an invariant stable holonomy for $\mathbb{P}(F)$. 

Observe that Proposition \ref{continuity} implies the continuity of invariant stable holonomies in compact parts of the strong-stable foliation. That is, the application $$(f,x,y)\mapsto H^s_{x,y}(f),$$ is continuous on $W^s_n(f)=\{(g,x,y) : g\in \mathcal{V}(f)\  \text{and}\  g^n(y)\in W^s_{loc}(g^n(x))\}$, for every $n\geq 1$.

There are analogous propositions and properties for the invariant unstable holonomy, $h^u_{x,y}$. Locally, it will be defined by the projectivization of $$H^u_{x,y}=\lim\limits_{n\rightarrow -\infty} F^n(y)^{-1}\circ P_{E^c(f^{n}(y))} \circ \pi^{n}_{x,y} \circ F^{n}(x),$$ where $$F^{-n}(x)=F^{-1}(f^{-n+1}(x))\circ \cdots \circ F^{-1}(x),$$ for every $n\in \mathbb{N}$.

Let $\zeta=[z_0,z_1,..,z_N]$ be a $su$-path for $f$ and denote $H_{z_i}=H^{*}_{z_{i-1},z_i}$ for every $i\in \{1,...,N\}$ with $*\in \{s,u\}$. Then, we consider $H_{\zeta}=H_{z_{N}}\circ \dots H_{z_1}$. The following Corollary gives an estimation about how this holonomy changes under the variation of $f$ and the $su$-path. 
\begin{cor}\label{compo} If $g$ is close enough to $f$, $\zeta_f=[x_0,...,x_N]$ and $\zeta_g=[y_0,...,y_N]$ are $su$-paths for $f$ and $g$ respectively, $a\in E^c(x_0,f)$ and $b\in E^c(y_0,g)$, then $$d(H_{\zeta_f}(a), H_{\zeta_g}(b))\leq \sum_{i=0}^{N-1} \prod_{j=i+2}^{N} \left\|H_{x_j}\right\| \psi(H_{x_{i+1}}) + \prod_{j=1}^{N} \left\|H_{x_j}\right\| d(a,b), $$ where $$\psi(H_{x_{i+1}})=d(H_{x_{i+1}}(a_i), H_{y_{i+1}}(\pi_{x_i,y_i}(a_i)))$$ and $$a_i=H_{x_i}\circ \dots \circ H_{x_1}(a).$$
\end{cor}
By Remark \ref{bound}, there exists $C>0$ such that for every $j=\{1,...,N\}$, if $x_{j-1}\in W^{*}_{loc}(x_j)$ with $*\,\{s,u\}$, then $\left\|H_{x_j}\right\|< 1+ C\, dist(x_{j-1},x_j)^{\alpha}$. Therefore, if $\zeta_f=[x_0,...,x_N]$ is a $su$-path with $x_{j-1}\in W^{*}_{loc}(x_j)$ and $dist(x_{j-1},x_j)< L$ for every $j\in \{1,...,N\}$, then $$\prod_{j=1}^{N} \left\|H_{x_j}\right\|< (1+ C\, L^{\alpha})^N.$$ 

This proves that we can find a bound for $\prod_{j=1}^{N} \left\|H_{x_j}\right\|$ depending only on the number of legs of the $su$-path and the distance between the nodes. This will be important in Section 6.

\section{Perturbation} 

In this Section, we construct a sequence of perturbations $f_k$ and study its properties. First, we state elementary results for $C^r$-perturbations in Lemma \ref{pert}. In Section 4.2, we find a $su$-path from $p$ to itself with slow recurrence, Proposition \ref{pivot}, and apply the previous Lemma to construct the perturbations $f_k$. This is done in Lemma \ref{pertfinal}. In Section 4.3 and 4.4, we study how the $su$-path and the center bundle change when we perturb the diffeomorphism. The main results are Proposition \ref{su} and Proposition \ref{angulos}. Finally, in Section 4.5 we summarize all the results to obtain Corollary \ref{holfinal} which gives estimations for the variation in the holonomies. 

\subsection{$C^r$-elementary perturbations}

Fixed $r\geq 2$, we are going to define the $C^r$ Whitney topology in the volume-preserving and symplectic case specifying basic neighborhoods.

Let $\mu$ be a volume form and pick two finite open coverings $\mathcal{U}=\{(U_j,\phi_j): j=1,....,n\}$ and $\mathcal{V}=\{(V_j, \psi_j): j=1,...,n\}$ of $M$ by $C^r$ conservative coordinates charts such that $f(\overline{U_j})\subset V_j$ for all $j$. This means that we are using \cite{M} to find $\phi_j:U_j\longrightarrow \mathbb{R}^{d}$ and $\psi_j:V_j\longrightarrow \mathbb{R}^{d}$, $C^r$ diffeomorphisms with $\mu=\phi_j^{*}(du_1\wedge \dots \wedge du_d)=\psi_j^{*}(du_1\wedge \dots \wedge du_d)$ where $(u_1, \dots, u_d)$ are coordinates in $\mathbb{R}^{d}$. 

Let $\epsilon>0$. Define $\mathcal{\eta}^r_{\mu}(f,\mathcal{U},\mathcal{V},\epsilon)$ to be the set of diffeomorphisms $g\in Diff^r_{\mu}(M)$ such that \begin{enumerate} [label=(\alph*)]
\item $g(\overline{U_j})\subset V_j$ for all $j$ and 
\item $\left\|\partial^{\iota} \psi_j g \phi_j^{-1}(x)-\partial^{\iota} \psi_j f \phi_j^{-1}(x)\right\|< \epsilon$ for $x\in \phi(U_j)$, $\left|\iota \right|\leq r$ and $j\in \{1,...,n\}$.
\end{enumerate}
Here $\iota=(\iota_1, .... , \iota_r)$ is a multi-index of non-negative integers, $\left|\iota \right|=\iota_1+....+\iota_r$, and $\partial^{\iota}$ denotes the corresponding partial derivative. 

For the symplectic case, pick two finite open coverings by $C^r$ symplectic charts. That is, use Darboux's Theorem to find $\phi_j:U_j\longrightarrow \mathbb{R}^{2d}$ and $\psi_j:V_j\longrightarrow \mathbb{R}^{2d}$, $C^r$ diffeomorphisms with $\omega=\phi_j^{*}(du\wedge dv)=\psi_j^{*}(du\wedge dv)$ where $\omega$ is the symplectic form and $(u,v)$ are coordinates in $\mathbb{R}^{2d}$. Then, $\mathcal{\eta}^r_{\omega}(f,\mathcal{U},\mathcal{V},\epsilon)$ is defined analogously.

We will write $\mathcal{\eta}^r_{*}(f,\mathcal{U},\mathcal{V},\epsilon)$ with $*\in \{\mu,\omega\}$.
\begin{lema}\label{pert}
Fix $r\geq 2$. Let $f$ be a partially hyperbolic volume-preserving (symplectic) diffeomorphism with $dim\; E^c=2$. Then, there exist $\epsilon_0>0$, $\delta_0>0$ and $C_0>0$ such that for every $0< \epsilon <\epsilon_0$, $0<\delta <\delta_0$ and $z\in M$, there exists $g\in \mathcal{\eta}^r_{*}(f,\mathcal{U},\mathcal{V},\epsilon)$ such that 
\begin{enumerate} [label=\emph{(\alph*)}]
\item $g(x)=f(x)$ if $x\notin B_{\delta}(z),$
\item $g(z)=f(z)$ and  
\item $Dg_z=Df_z \circ A_{\beta}$ where $sin\, \beta= C_0\, \delta^{r-1} \epsilon$ and $A_{\beta}$ is the linear map from $TM_{z}$ to $TM_{z}$ given in coordinates $TM=E^s\oplus E^c\oplus E^u$ by $$\begin{pmatrix}
Id_s & 0 & 0 \\
0  & R_{\beta} & 0 \\
0 & 0 & Id_u
\end{pmatrix} $$ with $Id_{**}: E^{**}_{z}\longrightarrow E^{**}_{z}$ being the identity map for $**\in \{s,u\}$ and $R_{\beta}$ the rotation of angle $\beta$ in some (symplectic) base $\{e_1,f_1\}$ of $E^c(z)$. 
\end{enumerate}
\end{lema} 

The proof in the volume-preserving setting is standard. However, we need to be more careful for the symplectic case and we need to use generating functions. The symplectic version of this Lemma is a direct consequence of Lemma 2.1 (Perturbation Lemma) of \cite{NEW}.

\subsection{Choice of the perturbation} 
Now, we use Lemma \ref{pert} to construct a sequence of perturbations $f_k$. As already mentioned, we need to find a $su$-path from $p$ to itself with slow recurrence. 

\begin{prop}[Proposition 8.2, \cite{ASV}]\label{pivot} Let $f$ be a partially hyperbolic accessible $C^2$ diffeomorphism. Then, for every $x\in M$ there exist a $su$-path, $\zeta=[z_0, ..., z_N]$ with $x=z_0=z_N$, $l\in \{0,...,N\}$ and $c>0$ such that $$dist (f^j(z_i), z_l)\geq \frac{c}{1+j^2},$$ for every $(j,i)\in \mathbb{Z}\times \{0,..., N\}\setminus (0,l)$.
\end{prop}
For every partially hyperbolic diffeomorphism $f$, there exist $R_1>0$ and a $C^1$ neighborhood of $f$, $\mathcal{V}(f)$, such that for every $g\in \mathcal{V}(f)$ and every $y\in M$, the ball $B(y,R_1)$ is contained in foliation boxes for both $W^u(g)$ and $W^s(g)$. See \cite{HPS}.
 
In the next section, we are going to need that $dist(z_{i-1},z_{i})<R_1$ for every $i\in \{1,...,N\}$. Since we are using the same definition of local strong-stable and local strong-unstable leaves than in Section 5 of \cite{ASV}, this will imply that $z_{i-1}\in W^{*}_{loc}(z_i)$ for every $i\in \{1,...,N\}$ with $*\,\in\{s,u\}$. 

It is possible to slightly modify the proof of Proposition \ref{pivot} in \cite{ASV} to control the distance between the nodes and obtain the desire bound. However, notice that we need to fix $f$ in order to obtain $R_1$ and only then apply the Proposition. 

We are going to use this remark in the next Section.

\begin{obs}\label{box} Fixed $f\in B^r_{*}(M)$ and $x\in M$, we can suppose that the $su$-path given by Proposition \ref{pivot} for $f$ and $x$, $\zeta=[z_0,...,z_N]$, satisfies $dist(z_{i-1},z_{i})<R_1$ for every $i\in \{1,...,N\}$.
\end{obs}

Let $*\in \{\mu, \omega\}$ and $r\geq 2$. Fix $f\in B^r_{*}(M)$ and suppose $p$ is a periodic point for $f$. Then, apply Proposition \ref{pivot} to $f$ and $p$. We are going to construct a sequence of perturbations for $f$, like in Lemma \ref{pert}, supported in the point $z_l$. 

First, we fix some constant $\sigma=\sigma(\upsilon, \alpha, n_p, N)>0$. Here, $\upsilon$ represents the functions in Equation (\ref{ph}) for $f$, $\alpha$ is the exponent for which $f$ is $\alpha$-pinched and $\alpha$-bunched, $n_p$ is the period of $p$ and $N$ the number of nodes in the $su$-path given by Proposition \ref{pivot}. This is a technical constant that we need to consider in order to have exponential estimations in Corollary \ref{holfinal}, Equation (\ref{ques}) and Equation (\ref{nodoscerca}).

Then, define 
\begin{equation}\label{deltak}
\delta_k=\frac{c}{1+(\sigma\,k)^2},
\end{equation}
for every $k\geq1$, where $c>0$ is given by Proposition \ref{pivot}.  

The following Lemma is a corollary of Lemma \ref{pert}. We are assuming that $\epsilon_0$ is small enough in order to have all the estimations in Section 3 uniform for every $g$ $\epsilon_0$-close to $f$.
\begin{lema}\label{pertfinal} Let $*\in \{\mu, \omega\}$ and $r\geq 2$. There exist $\epsilon_0>0$, $k_0\in \mathbb{N}$ and $C_0>0$ such that for any $0<\epsilon<\epsilon_0$ and $k\geq k_0$, there exists $f_k\in B^r_{*}(M)\cap \mathcal{\eta}^r_{*}(f,\mathcal{U},\mathcal{V},\epsilon)$ such that
\begin{enumerate}[label=(\alph*)]
\item $f_k(x)=f(x)$ if $x\notin B_{\delta_k}(z_l)$,
\item $f_k(z_l)=f(z_l)$ and
\item $Df_k(z_l)=Df(z_l) \circ A_{\beta_k}$ with $\sin\, \beta_k=C_0\, \delta_k^{r-1}\, \epsilon$.
\end{enumerate}
Moreover, if we fix $\epsilon>0$ and consider the sequence defined by $f_k$ for $k\geq k_0$, we have $f_k\rightarrow f$ in the $C^1$ topology when $k\to \infty$. 
\end{lema}
Observe that the sequence of perturbations $f_k$ is $\epsilon$-close to $f$ in the $C^r$-topology. However, since $\delta_k$ is going to zero as $k$ goes to infinity, the $C^1$ distance is also going to zero, because of that $f_k\rightarrow f$ in the $C^1$ topology. We are going to use this fact in Section 5 and 6.

\subsection{Control of the $su$-paths}

Proposition \ref{pivot} gives information about how much time the nodes of the $su$-path stay outside of the support of the perturbation. We will use this information to estimate how the dynamics is changing. Some results similar to these appear in \cite{DW}.

Let $*\in \{ \mu, \omega \}$, $r\geq 2$, $f\in B^r_{*}(M)$ and suppose $p$ is a periodic point for $f$. 

For the functions in Equation (\ref{ph}), define $$\nu(x,R)=\sup\limits_{y\in B(x, R)} \nu(y) \quad \quad \quad \gamma(x,R)=\inf\limits_{y\in B(x, R)} \gamma(y),$$ $$\widehat{\nu}(x,R)=\sup\limits_{y\in B(x, R)} \nu(y) \quad \text{and} \quad \widehat{\gamma}(x,R)=\inf\limits_{y\in B(x, R)} \gamma(y).$$

Then, by continuity of the functions and compactness of $M$, there exist $R_0>0$ and $\tau_0<1$ such that for every $x\in M$, $$\nu(x,R_0) < \tau_0\; \gamma(x,R_0)\quad \text{and} \quad \widehat{\nu}(x,R_0)< \tau_0\; \widehat{\gamma}(x,R_0).$$ 

\begin{obs}\label{box2} Observe that $R_0$ depends only on $f$ and therefore we can suppose that the constant $R_1$ in Remark \ref{box} was chosen to satisfy $R_1\leq R_0$.
\end{obs}

\begin{lema}\label{dw} Fix $\tau\in (\tau_0, 1)$. There exist $C_1>0$ and $\epsilon_1>0$ such that for every $x\in M$, $y\in W^s_{loc}(x)\cap B(x, R_0)$ and $g$ $\epsilon_1$-close to $f$ in the $C^1$ topology, if there exists $m\in \mathbb{N}$ with $f^j(x)=g^j(x)$ and $f^j(y)=g^j(y)$ for every $1\leq j\leq m$, then there exists $w\in W^s(x,g)\cap B(y, C_1\tau^{m})$. 
\end{lema}
\begin{proof}
Since $f$ is partially hyperbolic, there exists a cone family around $E^u\oplus E^c$, $K^{cu}$, such that
\begin{enumerate}[label=(\alph*)]
\item $Df(K^{cu}(x))\subset K^{cu}(f(x))$ for every $x\in M$,
\item $K^{cu}$ is uniformly transverse to $E^s$,
\item For every $v\in K^{cu}(x)$, 
$$\left\|Df_x(v)\right\|\geq \tau\,\nu(x,R_0) \left\|v\right\|,$$
for every $x\in M$. 
\end{enumerate}

Moreover, all the above are still valid for every $g$ $C^1$-close enough to $f$. 

Let $V$ be a topological disk of dimension $u+c$ passing through $y$ such that $TV\subset K^{cu}$. Since $f^{m}(x)=g^{m}(x)$ and $f^{m}(y)=g^{m}(y)$, we have $$dist(g^{m}(x), g^{m}(y))< \nu(x,R_0)^{m}.$$ Then, there exists $C_1>0$, depending only on $f$, and $w_1\in W^s(g^{m}(x),g)\cap g^{m}(V)$ such that $$dist(g^{m}(x), w_1)< C_1\: \nu(x,R_0)^{m}.$$ Define $w=g^{-m}(w_1)$, then $w\in W^s(x,g)$ and $$dist (y,w)< C_1\, \tau^{m}.$$
\end{proof} 
There is an analogous statement for the strong-unstable foliation $W^u$. Using these results we are able to prove the following Lemma.

\begin{lema}\label{nodos} 
If $\zeta=[z_0,...,z_N]$ is the $su$-path given by Proposition \ref{pivot} for $f$ and $p$ and $f_k$ is given by Lemma \ref{pertfinal} for some $\epsilon>0$, then there exist $C_1>0$, $\tau\in(0,1)$ and $k_1\in \mathbb{N}$ such that for every $k\geq k_1$ and $i\in \{1,... ,N\}$ there exist points $w_i^k\in M$ with $$w_i^k\in W^{*}(z_{i-1},f_k)\cap B(z_i, C_1\, \tau^{\sigma k}).$$  Here, the constant $\sigma$ is given by Equation (\ref{deltak}) and 
\begin{equation*}
W^{*}(z_{i-1},f_k)=\begin{cases}
     W^{s}(z_{i-1},f_k)&\text{if}\quad z_i\in W^s(z_{i-1},f),\\
		 W^{u}(z_{i-1},f_k)&\text{if}\quad z_i\in W^u(z_{i-1},f).
		\end{cases}
\end{equation*}
\end{lema}
\begin{proof}
Fix $i\in \{1,...,N\}$ and suppose $z_{i}\in W^s(z_{i-1},f)$. Notice we can assume $z_{i}\in W^s_{loc}(z_{i-1},f)\cap B(z_{i-1}, R_0)$. This is a consequence of Remark \ref{box} and Remark \ref{box2}. Then, we are going to apply Lemma \ref{dw} for $x=z_{i-1}$ and $y=z_i$. The value of $k_1$ is chosen in order to have $f_k$ $\epsilon_1$-close to $f$ in the $C^1$ topology for every $k\geq k_1$. 

By Equation (\ref{deltak}) and Lemma \ref{pertfinal}, $$f^j_k(z_{i-1})=f^j(z_{i-1}) \quad \text{and} \quad f^j_k(z_{i})=f^j(z_{i}),$$ for every $1\leq j\leq \sigma\,k.$ 

Then, for every $k\geq k_1$ there exists $w^k_i\in W^s(z_{i-1},f_k)\cap B(z_i, C_1\, \tau^{\sigma k})$. The case for $z_{i-1}\in W^u(z_i)$ is analogous.
\end{proof}

By the results in \cite{PSW1, PSW2}, the $\alpha$-pinched condition implies that the $W^s$ and $W^u$ holonomies are $\alpha$-H\"older. Moreover, there exists a $C^2$ neighborhood of $f$, $\mathcal{V}(f)$, such that the $W^s$ and $W^u$ holonomies for every $g\in \mathcal{V}(f)$ are $\alpha$-H\"older with uniform H\"older constant. For this see \cite{W}. 

We can suppose that $\epsilon_0>0$ in Lemma \ref{pertfinal} was chosen small enough in order to guarantee that $f_k\in \mathcal{V}(f)$.

Fix $k\geq k_1$ and define $z_1^k=w_1^k$ by Lemma \ref{nodos}. Then, $$dist(z^k_1, z_1) < C_1\, \tau^{\sigma k}.$$ Suppose $z_2\in W^s(z_1,f)$. Let $R_1>0$ be the constant in Remark \ref{box} and $w^k_2$ be defined by Lemma \ref{nodos}. If $k$ is big enough, $z^k_1, w^k_2\in B(z_1, R_1)$. Then, $z_2$, $z^k_1$ and $w^k_2$ are all in the same foliation box for $W^s(f_k)$. Denote $U$ this foliation box and let $\Sigma(x)$ be a smooth foliation by admissible transversals defined in $U$. Define $z^k_2$ as the only point of intersection of $W^s(z_1^k,f_k)$ with $\Sigma({w^k_2})$. Then, there exists $\widehat{C}_1=\widehat{C}_1(f)>0$ such that $$dist(w^k_2, z^k_2)<\widehat{C}_1\; dist(z_1^k,z_1)^{\alpha}.$$ Then, $$dist(z^k_2,z_2)< dist(z_2, w^k_2) + dist(w^k_2,z^k_2) < C_1\, \tau^{\sigma k} +  \widehat{C}_1\, C_1^{\alpha}\, \tau^{\sigma k \alpha}.$$ If $z_2\in W^u(z_1,f)$, we proceed in the same way using a foliation box for the strong-unstable foliation. 

Repeating the argument for all the nodes of $\zeta$ we have the following:
\begin{prop}\label{su} If $\zeta=[z_0,...,z_N]$ is the $su$-path given by Proposition \ref{pivot} for $f$ and $p$ and $f_k$ is given by Lemma \ref{pertfinal} for some $\epsilon>0$, then there exist $C_2>0$, $\tau\in (0,1)$ and $k_2\in \mathbb{N}$ such that for every $k\geq k_2$ there exists a $su$-path for $f_k$, $\zeta_k=[z^k_0,...,z^k_N]$, with $z^k_0=z_0=p$ and such that $$dist(z_i, z^k_i)< C_2\, \tau^{\sigma_1 k},$$ for every $i\in \{1,..., N\}$, where $\sigma_1=\sigma\,\alpha^{N}$ and $\sigma$ is given by Equation (\ref{deltak}).
\end{prop} 
Although $\zeta$ is a closed $su$-path, $z_0=z_N$, the $su$-path for $f_k$ given by this Proposition is not necessarily closed. We can have $z^k_0\neq z^k_N$.

\subsection{Angle estimations}

Now we study how the splitting in the tangent bundle is changing under the variation of the diffeomorphism. 

Let $V$ be a vector space with inner product and let $E_1$ and $E_2$ be subspaces of $V$. Define $\angle\; (E_1, E_2)=\max \{\xi_1, \xi_2\}$ where $$\xi_1=\sup\limits_{x\in E_1, x\neq 0}\; \; \inf\limits_{y\in E_2, y\neq 0} \angle \; x,y,$$ and $\xi_2$ is defined analogously, changing the places of $E_1$ and $E_2$.

The relation between this definition and the distance of subspaces defined in Section 3 is given by $$\sin \,\angle\, (E_1,E_2)=dist(E_1,E_2).$$

By Equation (\ref{ph}), there exist $C_3>0$ and $\theta_0>0$ such that for every $F^{u+c}$ and $F^{c+s}$ distributions of dimension $u+c$ and $c+s$ respectively, with
\begin{equation}\label{hipot}
\max\{ \angle(E^{u+c}, F^{u+c}), \angle (E^{c+s}, F^{c+s})\}\leq \theta_0,
\end{equation}
we have
\begin{equation}\label{unest}
\angle (Df_x^{j}(E_x^{u+c}), Df_x^{j}(F_x^{u+c})) \leq C_3\,\rho^j,
\end{equation} 
 and
\begin{equation}\label{est}
\angle (Df_x^{-j}(E_x^{c+s}), Df_x^{-j}(F_x^{c+s})) \leq C_3\,\rho^j.
\end{equation} 
for every $x\in M$ and $j\geq 0$, where $$\rho=\max\limits_{x\in M}\left( \max \left\{\nu(x)/\gamma(x) , \widehat{\nu}(x)/\widehat{\gamma}(x)\right\}\right).$$
\begin{lema}\label{p1p2} There exist $C_3>0$, $\rho\in (0,1)$ and $\epsilon_3>0$ such that if $g$ is $\epsilon_3$-close to $f$ in the $C^1$ topology, $g=f$ outside some compact set $I$ and there exists $m\in \mathbb{N}$ such that $f^j(x)$ do not enter $I$ for every $j\in \mathbb{Z}$ with $\left|j\right|\leq m$, then 
\begin{equation}\label{p1}
\angle (E^{u+c}(x,f), E^{u+c}(x,g) \leq C_3\,\rho^{m}
\end{equation}
and
\begin{equation}\label{p2}
\angle (E^{c+s}(x,f), E^{c+s}(x,g)) \leq C_3\,\rho^{m}
\end{equation} 
\end{lema} 
\begin{proof}
For every $j\in \mathbb{Z}$ with $\left|j\right|\leq m$, $f^j(x)\notin I$ implies that $f^j(x)=g^j(x)$ and $Df(f^j(x))=Dg(g^j(x))$. Moreover, if $g$ is $C^1$-close enough to $f$, the inequality in Equation (\ref{hipot}) holds for $F^{u+c}=E^{u+c}(g)$ and $F^{c+s}=E^{c+s}(g)$. Therefore,  
\begin{equation*}
\begin{aligned}
&\angle (E^{u+c}(x,f), E^{u+c}(x,g))\\
&=\angle (Df^{m}(E^{u+c}(f^{-m}(x),f)), Dg^{m}(E^{u+c}(g^{-m}(x),g)))\\
&=\angle (Df^{m}(E^{u+c}(f^{-m}(x),f)), Df^{m}(E^{u+c}(f^{-m}(x),g))).
\end{aligned}
\end{equation*}
Finally, we conclude Equation (\ref{p1}) by Equation (\ref{unest}). The stable case is analogous: we use Equation (\ref{est}) to prove Equation (\ref{p2}). \\
\end{proof} 
This provides a result about how the center bundle is changing with the perturbation. 
\begin{prop}\label{angulos} If $\zeta=[z_0,...,z_N]$ is the $su$-path given by Proposition \ref{pivot} for $f$ and $p$ and $f_k$ is given by Lemma \ref{pertfinal} for some $\epsilon>0$, then there exist $C_4>0$, $\rho\in (0,1)$ and $k_4\in \mathbb{N}$ such that for every $k\geq k_4$ and $i\in \{1,...,N\}$, we have
\begin{equation}\label{p3}
\angle (E^{c}(z_i,f), E^{c}(z_i,f_k)) \leq C_4\,\rho^{\sigma k-1},
\end{equation} 
where $\sigma$ is given by Equation (\ref{deltak}).
\end{prop}
\begin{proof}
If $i\neq l$, that is, $z_i$ is not the point where we did the perturbation, then by Equation (\ref{deltak}) and Lemma \ref{pertfinal} we have that $f^j(z_i)$ do not enter the support of the perturbation for every $j\in \mathbb{Z}$ with $\left|j\right|\leq \sigma k - 1$. Then, we can apply the Lemma above for $E^{u+c}$ and $E^{c+s}$ with $m=\sigma k - 1$. The Proposition follows because $E^c$ is the intersection of this two transversal bundles.
 
In order to prove Equation (\ref{p3}) for $z_l$, we need to prove a Lemma similar to Lemma \ref{p1p2} for $E^{u+c}(z_l)$ and $E^{c+s}(z_l)$, using the fact that $Df_k(z_l)=Df(z_l)\circ A_{\beta_k}$ and $A_{\beta_k}$ leaves invariant the subbundles in the splitting $E^u_{z_l}\oplus E^c_{z_l}\oplus E^s_{z_l}$. Then, we conclude the Proposition with the same argument than for the other nodes.   
\end{proof} 

\subsection{Summary of the results}

Let $*\in \{\mu, \omega\}$, $r\geq 2$, $f\in B^r_{*}(M)$ and suppose $p$ is a periodic point for $f$. Proposition \ref{pivot} gives a $su$-path from $p$ to itself, $\zeta=[z_0,...,z_N]$, and a node with slow recurrence $z_l$.

Let $f_k$ be given by Lemma \ref{pertfinal} for some $\epsilon>0$. Then, $f_k$ satisfies Propositions \ref{su} and \ref{angulos} for every $k\geq \max \{k_2,k_4\}$. Moreover, if $k$ is big enough, $f^j(p)\notin B_{\delta_k}(z_l)$ for every $j\in \{0,...,n_p-1\}$.

Let $$\zeta_1=[z_0,...,z_l]\quad \text{and} \quad \zeta_2=[z_N,..., z_l].$$ We can suppose $z_{l-1}\in W^s(z_l)$ and $z_l\in W^u(z_{l+1})$. 

In the notation of Proposition \ref{su}, define $$z_k=z^{k}_l, \quad p_k=z^{k}_N,$$ $$\zeta^k_1=[p,...,z_k]\quad \text{and} \quad \zeta^k_2=[p_k,...,z_k].$$  

In the following, for $i\in \{1,2\}$ $H_{\zeta_i}$ will denote the holonomy defined by $\zeta_i$ for $F=Df\vert E^c(f)$ and $H_{\zeta^k_i}$ the holonomy defined by $\zeta^k_i$ for $F_k=Df_k\vert E^c(f_k)$. Then, $$H_{\zeta_1}:E^c(p)\to E^c(z_l, f), \quad \quad H_{\zeta_2}:E^c(p)\to E^c(z_l, f),$$ $$H_{\zeta_1^k}:E^c(p)\to E^c(z_k, f_k)\quad \text{and} \quad H_{\zeta_2^k}:E^c(p_k, f_k)\to E^c(z_k, f_k).$$

Using Corollary \ref{compo} combined with Propositions \ref{su} and \ref{angulos} we can estimate the variation in the holonomies:
\begin{cor}\label{holfinal} There exist $C>0$, $\lambda\in (0,1)$ and $K\in \mathbb{N}$ such that for every $k\geq K$, $a\in E^c(p)$ and $a_k\in E^c(p_k,f_k)$ we have
\begin{equation}\label{rot}
d( R_{\beta_k}^{-1}\circ H_{\zeta_1}(a), H_{\zeta_1^k}(a))\leq C\, \lambda^k,
\end{equation}
and
\begin{equation}\label{direc}
d(H_{\zeta_2}(a), H_{\zeta_2^k}(a_k))\leq C\, \lambda^k + C\, d(a, a_k),
\end{equation}
where $R_{\beta_k}:E^c(z_l,f)\to E^c(z_l,f)$ is the rotation of angle $\beta_k>0$ defined by Lemma \ref{pertfinal}.
\end{cor}
\begin{proof}
Consider Equation (\ref{direc}). By Corollary \ref{compo}, it is enough to estimate the distances between the holonomies $H_{i+1}(f)$ and $H_{i+1}(f_k)$ for $i\in \{l,... ,N-1\}$, where $H_{i}=H^{**}_{w_{i+1},w_{i}}$ and $**\in \{s,u\}$. 

Fix $i\in \{l,...,N-1\}$ and suppose $z_{i+1}\in W^s(z_{i})$. The other case is analogous. 

By Proposition \ref{continuity}, we can estimate the distance between $H_{i+1}(f)$ and $H_{i+1}(f_k)$ by an expression depending on: $$\nu^k(z_i)^{\alpha(1-\varsigma)}, \quad \nu^k(z_i^k)^{\alpha(1-\varsigma)},$$ 
$$dist(z_i, z_i^k),\quad dist(z_{i+1}, z_{i+1}^k),$$  
$$\angle (E^c(z_i,f), E^c(z_i,f_k)), \quad \angle (E^c(z_{i+1},f), E^c(z_{i+1},f_k)),$$  
and the following terms: $$\left\|Df(f^j(z_{i-1}))-Df_k(f^j(z_{i-1}))\right\| \quad \text{and} \quad \left\|Df^{-1}(f^{j+1}(z_{i}))-Df_k^{-1}(f^{j+1}(z_{i}))\right\|,$$ for every $j=0,...,k$. 

Since $\nu<1$ and $\varsigma<1$, the first two terms are going exponentially fast to zero as $k\to \infty$. The estimations for the two terms in the second line are given by Proposition \ref{su} and for the third line by Proposition \ref{angulos}. By Equation (\ref{deltak}) and Lemma \ref{pertfinal} the last terms are all equal to zero because $f$ and $f_k$ coincide outside the support of the perturbation. 

Although we have exponential estimations for each term, we can only guarantee that the whole expression is going exponentially fast to zero as $k\to \infty$ due to the constant $\sigma$ in Equation (\ref{deltak}). 

If we want to estimate $d(H_{\zeta_1}(a), H_{\zeta_1^k}(a))$ everything works the same than above except that now we have the following term which is not zero, $$\left\|Df^{-1}(f(z_l))-Df_k^{-1}(f(z_l))\right\|.$$ Because of that we consider $R_{\beta_k}^{-1}\circ H_{\zeta_1}$ instead of $H_{\zeta_1}$ in Equation (\ref{rot}). This allow us to obtain the desired estimation. 

Notice that we are comparing the values of the derivative at the point $z_l$ and not in $z_k=z_l^k$. This is an important observation since we have no control on $Df_k(z_l^k)$. 
\end{proof}

\section {Accessibility}

We obtain a continuity property for $su$-paths under the variation of the diffeomorphism using the results and techniques in \cite{AV2}. In order to clarify the presentation we state here the results that we are going to need.

In this section, all the maps will be $C^1$ and we will always consider the $C^1$ topology. If $f$ is a partially hyperbolic diffeomorphism, we denote $u=dim\; E^u$, $s=dim\; E^s$. 

Recall that given two points $x,y\in M$, $x$ is \textit{accessible} from $y$ if there exist a path that connects $x$ to $y$, which is a concatenation of finitely many subpaths, each of which lies entirely in a single leaf of $W^u$ or a single leaf of $W^s$. This is a equivalence relation and we say that $f$ is \textit{accessible} if $M$ is the unique accessibility class. 

In the sequel, we state the principal results in \cite{AV2} which will allow us to prove the main results in this Section. The following Theorem was already mentioned in Section 2. 
\begin{teo}\label{avvi} If $f$ is a partially hyperbolic accessible diffeomorphism and the center bundle $E^c$ is 2-dimensional, then $f$ is stably accessible. 
\end{teo}
 The next Theorem provides a parametrization for accessibility classes. 
\begin{teo}\label{pls}  For every partially hyperbolic diffeomorphism $f:M\longrightarrow M$, there exist $l\geq 1$, a neighborhood of $f$, $\mathcal{V}(f)$, and a sequence $P_m: \mathcal{V}(f)\times M \times \mathbb{R}^{l(u+s)m}\longrightarrow M$ of continuous maps such that, for every $(g,z,v)\in \mathcal{V}(f)\times M \times \mathbb{R}^{l(u+s)m}$,
\begin{enumerate}[label=\emph{(\alph*)}]
\item $P_n(g,P_m(g,z,v),w)=P_{n+m}(g,z, (v,w))$ for every $w\in \mathbb{R}^{l(u+s)n}$;
\item $\xi\mapsto P_m(g,\xi,v)$ is a homeomorphism from $M$ to $M$ and $P_m(g, * , 0)=id$;
\item $\bigcup_{m\geq 0} P_m( \{(g,z)\}\times \mathbb{R}^{l(u+s)m})$ is the $g$-accessibility class of $z$.
\end{enumerate}
\end{teo}
Using this Theorem, Avila and Viana introduce a class of paths, called \textit{deformation paths}, contained in accessibility classes and having a useful property of persistence under the variation of the diffeomorphism and the base point. More precisely,
\begin{defn} A deformation path based on $(f,z)$ is a map $\gamma:[0,1]\longrightarrow M$ such that there exist $m\geq 1$ and a continuous map $\Gamma\mapsto \mathbb{R}^{l(u+s)m}$ satisfying $\gamma(t)=P_m(f,z, \Gamma(t))$. 
\end{defn}
The continuity of the sequence of maps $P_m$ given by Theorem \ref{pls} implies the following Corollary.
\begin{cor}\label{nuevo} If $\gamma$ is a deformation path based on $(f,z)$, then for every $g$ close to $f$ and any $w$ close to $z$, there exists a deformation path based on $(g,w)$ that is uniformly close to $\gamma$.
\end{cor}
The main technical step in the proof of Theorem \ref{avvi} is a result of approximation of general paths in accessibility classes by the deformation paths defined above. We state a simpler version of this result that is sufficient for our purposes.
\begin{teo}\label{den} If $f$ is accessible, then for every $z\in M$ the set of deformation paths based on $(f,z)$ is dense on $C^0([0,1], M)$. 
\end{teo} 
The final ingredient is what it is called Intersection Property, and it is in this result where the hypothesis of $dim\; E^c=2$ is necessary.
\begin{teo} [Intersection Property] \label{int} Let $f$ be a partially hyperbolic diffeomorphism with 2-dimensional center bundle. Let $D$ be a 2-dimensional disk transverse to $E^s\oplus E^u$ and $\eta_u$, $\eta_s$ be smooth paths in $D$ intersecting transversely at some point. Then, for every diffeomorphism $g$ $C^1$-close to $f$ and any continuous paths $\gamma_u$, $\gamma_s$ uniformly close to $\eta_u$, $\eta_s$, there are points $x_u$, $x_s$ in the images of $\gamma_u$, $\gamma_s$ such that $W^u(x_u,g)$ intersects $W^s(x_s,g)$. 
\end{teo}

Let $f$ be a partially hyperbolic accessible diffeomorphism with 2-dimensional center bundle and $x,y\in M$. By Theorem \ref{avvi}, if $g$ is $C^1$-close enough to $f$, then there exists some $su$-path for $g$ joining $x$ to $y$. Besides, the proof of the Theorem in \cite{AV2} uses Theorem \ref{den} and \ref{int} and provides a way to find the $su$-path for $g$. The following results uses that information to prove relations between the $su$-paths for $f$ and for a sequence $f_k\rightarrow f$.
\begin{prop} \label{caminocont}
Let $f$ be a partially hyperbolic accessible diffeomorphism with 2-dimensional center bundle. For every $x,y\in M$, $y_k\rightarrow y$ and every sequence $f_k\rightarrow f$ in the $C^1$ topology, there exist a subsequence $k_j$, su-paths for $f_{k_j}$ denoted by $\zeta_{k_j}$ and a su-path for $f$ denoted by $\zeta$ satisfying the following: 
\begin{enumerate}[label=\emph{(\alph*)}] 
\item $\zeta_{k_j}=[z_0^j, ... ,z_N^j]$ joins $x$ to $y_{k_j}$,
\item $\zeta=[z_0, ..., z_N]$ joins $x$ to $y$ and 
\item for every $\epsilon>0$ there exists $K\in \mathbb{N}$ such that for every $k_j\geq K$, $$dist (z_i, z_i^j)< \epsilon$$ for every $i\in \{0,...,N\}$.
\end{enumerate}
\end{prop}
\begin{proof}
Fix a small 2-disk $D$ and $\eta_u$ and $\eta_s$ like in the hypotheses of Theorem \ref{int}. By Theorem \ref{den}, there exist a deformation path $\gamma_u$ based on $(f,x)$ which is uniformly close to $\eta_u$ and a deformation path $\gamma_s$ based on $(f,y)$ which is uniformly close to $\eta_s$. Moreover, by Corollary \ref{nuevo}, if $k$ big enough, there exist
\begin{enumerate}[label=(\alph*)]
\item a deformation path $\gamma^k_u$ based on $(f_k,x)$ which is still close to $\eta_u$ and 
\item a deformation path $\gamma^k_s$ based on $(f_k,y_k)$ which is still close to $\eta_s$.
\end{enumerate}
Theorem \ref{int} implies that there exist $t^k_u, t^k_s\in [0,1]$ and $w_k\in M$ such that $$w_k\in W^u(\gamma^k_u(t^k_u),f_k)\cap W^s(\gamma^k_s(t^k_s),f_k).$$
Then, for every $k$ big enough, we have a $su$-path for $f_k$, joining $x$ to $y_k$, denoted by $\zeta_k$ and defined by the nodes of $\gamma_u^k(t_u^k)$, the intersection point $w_k$ and the nodes of $\gamma_s^k(t_s^k)$. 

By compactness, there exist a subsequence $k_j$ and $t_u$, $t_s$ such that $$t^{k_j}_{u}\rightarrow t_u\quad \text{and} \quad t^{k_j}_{s}\rightarrow t_s.$$

Since $f_k\rightarrow f$ and $W^s$ and $W^u$ are continuous under the variation of the diffeomorphism we can find $w\in M$ such that $w_k\rightarrow w$ and $$w\in W^u(\gamma_u(t_u),f)\cap W^s(\gamma_s(t_s),f).$$ 

Denote $\zeta$ the $su$-path for $f$ joining $x$ to $y$ and defined by the nodes of $\gamma_u(t_u)$, the intersection point $w$ and the nodes of $\gamma_s(t_s)$.

Finally, by the construction of the $su$-paths and, again, Corollary \ref{nuevo}, we have that for every $\epsilon>0$ there exists $K\in \mathbb{N}$ such that the distance between the nodes of $\zeta$ and $\zeta_{k_j}$ is bounded by $\epsilon$ for every $k_j\geq K$.
\end{proof} 
The $su$-paths in the Proposition above can be chosen in a uniform way. That is, with a uniform number of legs and a uniform bound for the distance between the nodes.
\begin{cor}\label{uniform} Let $f$ be a partially hyperbolic accessible diffeomorphism with 2-dimensional center bundle. Then, there exist $L>0$ and $N>0$ such that for every $x,y\in M$, $y_k\rightarrow y$ and every sequence $f_k\rightarrow f$ in the $C^1$ topology, the su-paths defined by Proposition \ref{caminocont} can be taken to have at most $N$ legs and distance between the nodes bounded by $L$.
\end{cor}
\begin{proof}
Observe that in order to prove this Corollary it is sufficient to prove the following claim.
\begin{af} Fix $\eta\in C^0([0,1], M)$ and $\epsilon>0$. Then, there exist $L>0$ and $N>0$ such that for every $x\in M$ there exists a deformation path based on $(f,x)$, denoted by $\gamma$, which is $\epsilon$-close to $\eta$ and satisfies that for every $t\in [0,1]$, the su-path defined by $\gamma(t)$ has at most $N$ legs and the distance between the nodes is bounded by $L$. 
\end{af}
By Theorem \ref{den}, for every $x\in M$ there exists a deformation path based on $(f,x)$, that is $\epsilon$-close to $\eta$. Therefore, the claim follows from the persistence of the deformations under the variation of the base point and the compactness of $M$.
\end{proof}
Corollary \ref{uniform}, together with Remark \ref{bound} and Corollary \ref{compo} give the following result.
\begin{cor}\label{unifbound} Let $*\in \{\mu, \omega\}$, $r\geq 2$ and $f\in B^r_{*}(M)$. Then, there exists $C>0$ such that for every $x,y\in M$, $y_k\rightarrow y$ and every sequence $f_k\rightarrow f$ in the $C^1$ topology, the su-paths given by Proposition \ref{caminocont}, denoted by $\zeta_{k_j}$ and $\zeta$, can be taken to satisfy the following estimation for the holonomies defined by them,
$$d(H_{\zeta}(a), H_{\zeta_{k_j}}(b))\leq \psi(k_j) + C \,d(a,b), $$
where $\psi(k_j)$ goes to zero as $k_j$ goes to $\infty$.
\end{cor}
There are analogous estimations for $h_{\zeta}=\mathbb{P}(H_{\zeta})$ and $h_{\zeta_{k_j}}=\mathbb{P}(H_{\zeta_{k_j}})$.  We are going to use this result to prove Proposition \ref{soporte} and to conclude the proof of Theorem B in Section 7.

\section{Disintegration}

Let $*\in \{\mu, \omega\}$, $r\geq 2$, $f\in B^r_{*}(M)$ and assume $p$ is a pinching periodic point for $f$. That is, $Df^{n_p}|E^c(p)$ has two real eigenvalues with different norms, where $n_p=per(p)$. Then, there exist $C_1>0$, $\theta_0>0$, $\rho\in (0,1)$ and one-dimensional subspaces $E_1$, $E_2$ of $E^c_p$ such that for every $F_1$ and $F_2$ one-dimensional subspaces of $E^c_p$ with 
\begin{equation}\label{teta}
\max\{ \angle(E_1,F_1), \angle(E_2, F_2)\}<\theta_0,
\end{equation}
we have
\begin{equation}\label{frente}
\angle( Df^{n_pj}(E_1), Df^{n_pj}(F_1))\leq C_1\, \rho^j
\end{equation} 
 and
\begin{equation}\label{atras}
\angle( Df^{-n_pj}(E_2), Df^{-n_pj}(F_2))\leq C_1\, \rho^j,  
\end{equation} 
for every $j\geq 0$.

By the Invariance Principle, if $\lambda^c_1(f)=\lambda^c_2(f)$ almost everywhere, then every $\mathbb{P}(F)$-invariant probability measure $m$ with $\pi_{*}m=\mu$ admits a disintegration, $\{m_x: x\in M\}$, invariant by holonomies and continuous with the $\text{weak}^*$ topology. The continuity of $m_x$ and the invariance of $m$ implies that $\mathbb{P}(F(x))_{*}m_{x}=m_{f(x)}$ for every $x\in M$. 

Then, if $a,b\in \mathbb{P}(E^c_p)$ are defined by $a=[E_1]$ and $b=[E_2]$, we have 
\begin{equation}\label{period}
supp\; m_p\subset \{a,b\}.
\end{equation}

Let $f_k$ and $k_0$ be given by Lemma \ref{pertfinal} for some $\epsilon>0$. Suppose $\lambda^c_1(f_k)=\lambda^c_2(f_k)$ almost everywhere for every $k\geq k_0$. We will denote $F_k=Df_k|E^c(f_k)$ the center derivative cocycle and $\mathbb{P}(F_k)$ its projectivization. 

We can suppose that for every $k\geq k_0$, $f^j(p)\notin B_{\delta_k}(z_l)$ for every $j\in \{0,...,n_p-1\}$. Then, for every $k\geq k_0$ and any $\mathbb{P}(F_k)$-invariant probability measure $m^k$ with $\pi_{*}m^k=\mu$, we have $$supp\; m^k_p\subset \{a,b\}.$$ 

Moreover, if for every $k\geq k_0$ we fix some $m^k$, then there exist a subsequence $k_j$ and a measure $m$ in $\mathbb{P}(TM)$ such that $m^{k_j}\rightarrow m$ in the weak$^*$ topology. The limit measure $m$ has the following properties: 
\begin{enumerate}[label=(\alph*)]
\item $supp\; m \subset \mathbb{P}(E^c(f))$,
\item $m$ projects down to $\mu$,
\item $m$ is $\mathbb{P}(F)$-invariant.
\end{enumerate}
Denote $m^{k_j}_{p}$ and $m_{p}$ the element of the disintegration given by the Invariance Principle at $p$ for $m^{k_j}$ and $m$ respectively.  
\begin{prop}\label{soporte} If $\vert supp\, m_p \vert=1$, then there exist a subsequence of $k_j$, that we continue to denote $k_j$, and $K_0\in \mathbb{N}$ such that $$supp\; m_p\subset supp\; m^{k_j}_p,$$ for every $k_j\geq K_0$. 
\end{prop}
\begin{proof} 
Suppose that $supp\; m_p=\{a\}$. The case $supp\; m_p=\{b\}$ is analogous. 

Consider $C>0$ given by Corollary \ref{unifbound} and fix some $0 <\delta< d(a,b)/4C$.

Define the function $\xi:M\longrightarrow \mathbb{P}(\mathcal{V})$ by $\xi(x)=(x,supp\; m_x)$ and the set $$T_{\delta}=\left\{(x,v)\in \mathbb{P}(TM): (x,v)\in B_{\delta}(\xi(x))\right\}.$$ The Invariance Principle implies that the function $\xi$ is continuous and therefore $T_{\delta}$ is an open set. Moreover, by definition, $m(T_{\delta})=1$. These two properties imply that $$m^{k_j}(T_{\delta})=\int m^{k_j}_x(T_{\delta}\cap \mathbb{P}(E^c(x,f_{k_j}))) d\mu(x)\rightarrow 1.$$ Then, there exist a subsequence of $k_j$, that we continue to denote $k_j$, $x\in M$ and $K_1\in \mathbb{N}$ such that for every $k_j\geq K_1,$ $$T_{\delta}\cap supp\;m^{k_j}_{x}\neq \emptyset.$$  

We apply Proposition \ref{caminocont} to $f_{k_j}$, $x$ and $y_{k_j}=p$. Then, we have a new subsequence, that we continue to denote $k_j$, $su$-paths for $f_{k_j}$ denoted by $\zeta_{k_j}$ and a $su$-path for $f$ denoted by $\zeta$, all joining $x$ to $p$. 

Denote $h$ the holonomy defined by $\zeta$ for $\mathbb{P}(F)$ and $h_{k_j}$ the holonomy defined by $\zeta_{k_j}$ for $\mathbb{P}(F_{k_j})$. By Corollary \ref{unifbound}, there exist $C>0$ and a function $\psi(k_j)$ going to zero as $k_j\to \infty$, such that for every $a'\in \mathbb{P}(E^c(x,f))$ and every $b'\in \mathbb{P}(E^c(x,f_{k_j}))$, we have $$d(h(a'), h_{k_j}(b'))< \psi(k_j) + C\,d(a', b').$$
 
Since the disintegration given by the Invariance Principle is invariant by holono\-mies and we suppose $supp\; m_p=\{a\}$, we have $supp\; m_{x}=h^{-1}(a)$. Moreover, since $$supp\; m^{k_j}_{x}\cap T_{\delta}\neq \emptyset,$$ there exists $$a'_{k_j}\in supp\, m^{k_j}_{x}\quad \text{with} \quad d(h^{-1}(a), a'_{k_j})<\delta \quad \text{for every} \; k_j\geq K_1.$$

Define $a_{k_j}=h_{k_j}(a'_{k_j})$. Then, $a_{k_j}\in supp\; m^{k_j}_p$ and for $k_j$ big enough, $$d(a, a_{k_j})=d(h(h^{-1}(a)), h_{k_j}(a'_{k_j}))< \psi(k_j) + C\, d(h^{-1}(a), a'_{k_j})<d(a,b)/2.$$ Since $supp\, m^{k_j}_p\subset \{a,b\}$, this implies $a_{k_j}=a$ and finish the proof. 
\end{proof}

\section{Proof of the Theorems}

In this Section we give the details of the proof of Theorem B and explain how to conclude Theorem A. 

\subsection{Proof of Theorem B}

\begin{defn*} Let $f$ be a partially hyperbolic diffeomorphism and $p$ a periodic point with $n_p=per(p)$. We say that $p$ is a \emph{pinching periodic point} if $Df^{n_p}|E^c(p)$ has two real eigenvalues with different norms.
\end{defn*} 
\begin{teocon} Let $*\in \{\mu,\omega\}$, $r\geq 2$, $f\in B^{r}_{*}(M)$ and assume $f$ has a pinching periodic point, then $f$ can be $C^r$-approximated by volume-preserving (symplectic) diffeomorphisms whose center Lyapunov exponents are different. 
\end{teocon}
\begin{proof}
Let $f\in B^r_{*}(M)$ and $p$ be a pinching periodic point, with $n_p=per(p)$. Consider $F=Df\vert E^c$ and suppose $\lambda^c_1(f)=\lambda^c_2(f)$ almost everywhere.

We proved in the previous section (Equation (\ref{period})) that there exist $a,b\in \mathbb{P}(E^c_p)$ such that for every $\mathbb{P}(F)$-invariant probability measure $m$ with $\pi_{*}m=\mu$, the element of the disintegration given by the Invariance Principle at $p$ satisfies, $$supp\; m_p\subset \{a,b\}.$$

Fix a neighborhood in the $C^r$-topology, $\eta^r_{*}(f, \mathcal{U}, \mathcal{V}, \epsilon)$, with $\epsilon>0$ small enough.  

Consider the $su$-path given by Proposition \ref{pivot} for $f$ and $p$ and let $C_0>0$, $k_0\in \mathbb{N}$ and $f_k$ be defined by Lemma \ref{pertfinal} for this $\epsilon$. Then, for every $k\geq k_0$, $f_k$ has the following properties:
\begin{enumerate} [label=(\alph*)]
\item $f_k\in B^r_{*}(M)\cap\mathcal{\eta}^r_{*}(f,\mathcal{U}, \mathcal{V},\epsilon)$,
\item $f_k(x)=f(x)$ if $x\notin B_{\delta_k}(z)$
\item $f_k(z)=f(z)$ and 
\item $Df_k(z)=Df(z) \circ A_{\beta_k}$ with $\sin\, \beta_k=C_0\, \delta_k^{r-1}\, \epsilon$.
\end{enumerate}  
Moreover, $f_k\rightarrow f$ in the $C^1$ topology when $k\to \infty$.

Denote $F_k=Df_k\vert E^c_k$ and suppose $\lambda^c_1(f_k)=\lambda^c_2(f_k)$ almost everywhere. 

We can assume that for every $k\geq k_0$, $f^j(p)\notin B_{\delta_k}(z_l)$ for every $j\in \{0,...,n_p-1\}$. Then, for any $\mathbb{P}(F_k)$-invariant probability measure $m^k$ with $\pi_{*}m^k=\mu$, we have $$supp\, m^k_p\subset \{a,b\}.$$ 

Consider the following two cases: 
\begin{enumerate}[label=(\roman*)]
\item $\forall\: k\geq k_0$ there exists a $\mathbb{P}(F_k)$-invariant probability measure $m^k$ with $\pi_{*}m^k=\mu$ such that $supp\, m^k_p=\{a,b\}$.
\item $\forall\; \mathbb{P}(F)$-invariant probability measure $m$ with $\pi_{*}m=\mu$, we have $$\vert supp\, m_p \vert=1.$$
\end{enumerate}

It is enough to prove Theorem B for these cases. Suppose the sequence of perturbations $f_k$ does not satisfy the condition in case (i). Then, there exists $k_1\geq k_0$ such that for every $\mathbb{P}(F_{k_1})$-invariant probability measure $m^{k_1}$ with $\pi_{*}m^{k_1}=\mu$, we have $\vert supp\, m^{k_1}_p \vert=1.$ Then, $f_{k_1}$ is a diffeomorphism satisfying case (ii). If Theorem B is true in this case, we will be able to find a diffeomorphism $g$ which is $\epsilon$-close to $f_{k_1}$ in the $C^r$ topology and such that $\lambda^c_1(g)\neq \lambda^c_2(g)$ at almost every point. Moreover, since $f_{k_1}$ is $\epsilon$-close to $f$, $g$ will be $2\,\epsilon$-close to $f$. This will prove the Theorem.

In the following, we give the details of the proof of Theorem B for case (i) and (ii). Since the two arguments are very similar, we are going to explain them simultaneously. 

Let $K\in \mathbb{N}$ be given by Corollary \ref{holfinal}. If we are in case (i), for every $k\geq K$, there exists $m^k$ such that $supp\, m^k_p=\{a,b\}$. Then, via a subsequence, we can suppose that there exists a $\mathbb{P}(F)$-invariant probability measure $m$ with $\pi_{*}m=\mu$ such that $m=\lim m^k$. If we are in case (ii), for every $k\geq K$, we choose any $\mathbb{P}(F_k)$-invariant probability measure $m^k$ and define $m=\lim m^k$. Then, $\vert supp\, m_p\vert=1$ and we can apply Proposition \ref{soporte}. From now on, if we are in case (ii), it is understood that $f_k$ denotes the subsequence that verifies Proposition \ref{soporte} and $K\geq \widehat{K}$.

For every $k\geq K$, let $\{m^k_{x} : x\in M\}$ and $\{m_x: x\in M\}$ be the disintegrations given by the Invariance Principle for $m^k$ and $m$ respectively. Suppose $a\in supp\, m_p$, the other case is analogous. 

By Proposition \ref{su} there exist $C>0$, $\tau\in (0,1)$ and $p_k=z^{k}_N$ such that $$dist(p,p_k)<C\, \tau^{\sigma_1 k}.$$

The following argument allow us to find a point in the $supp\, m^k_{p_k}$ which is exponentially close to $supp\, m^k_p$. The hypothesis of $p$ being a pinching periodic point is essential here. We need this estimation in order to apply Corollary \ref{holfinal} and get a contradiction. 

Define $$q_k=f_k^{-n_p k}(p_k).$$ Then, there exists $C_2>1$ such that 
\begin{equation}\label{ques}
dist(p,q_k)\leq C_2^{-n_p k} dist(p, p_k) \leq C\, (C_2^{-n_p} \tau^{\sigma_1})^{k}.
\end{equation} Here $C_2$ depends on the functions in Equation (\ref{ph}) and $\sigma_1=\sigma\,\alpha^{N}$. The constant $\sigma$ is defined in Equation (\ref{deltak}) and was chosen in order to have this expression going to zero as $k\to \infty$. 
\begin{af} Let $\theta_0>0$ be the constant defined in Equation (\ref{teta}). Then, there exist $K_1\in \mathbb{N}$ such that for every $k\geq K_1$, there exists $d_{k}\in supp\, m^{k}_{q_{k}}$ with $d(a, d_{k})< \theta_0$. 
\end{af} 
\begin{proof} 
Since $q_k\rightarrow p$, we can apply Theorem \ref{caminocont} for $f_{k}$, $x=p$ and $y_k=q_{k}$. Then, there exist $su$-paths for $f_{k}$ denoted by $\zeta_{k}$ joining $p$ to $q_k$ and a $su$-path for $f$ denoted by $\zeta$ joining $p$ to $p$. Moreover, they satisfy Corollary \ref{unifbound}.

Denote by $h$ the holonomy defined by $\zeta$ for $\mathbb{P}(F)$ and $h_{k}$ the holonomy defined by $\zeta_k$ for $\mathbb{P}(F_k)$.

If we are in case (i), we have to possibilities: $h(a)=a$ or $h(a)=b$. Then, define 
\begin{equation*}
d_{k}=\begin{cases}
       h_{k}(a) &\text{if}\; h(a)=a,\\
			 h_{k}(b) &\text{if}\; h(a)=b. 
\end{cases}
\end{equation*}
Since $supp\, m^k_p=\{a,b\}$, in any case we have $d_k\in supp\, m^k_{q_k}$.

If we are in case (ii), then $h(a)=a$ and we define $d_k=h_{k}(a)$. Proposition \ref{soporte} implies that $d_k\in supp\, m^k_{q_k}$.

Then, by Corollary \ref{unifbound} there exist $\psi(k)\to 0$ as $k\to \infty$, such that $$d(a, d_k)\leq \psi(k).$$ Therefore, choose $K_1$ big enough in order to have $\psi(k)<\theta_0$ for every $k\geq K_1$.
\end{proof}
As before, $P_{E}$ will denote the orthogonal projection to $E$ and $\pi_{x,y}$ the parallel transport between $x$ and $y$. 

Define $$a_k= \left[P_{E^c_p} \circ \pi_{q_k, p} (d_k)\right],$$ where $[*]$ denote the class in the projective space, and $$c_k=\mathbb{P}(F^{n_p}_k(q_k)) (d_k).$$ 
We have the following consequences: $$a_k\in \mathbb{P}(E^c_p), \quad \angle(a, a_k)<\theta_0\quad \text{and} \quad c_k\in supp\; m^k_{p_k}.$$ 
Then, there exist $C_1>0$, $\rho\in (0,1)$ and $C_3>1$ such that 
\begin{equation}\label{nodoscerca}
\begin{aligned}
d(a,c_k)&=d(\mathbb{P}(F^{n_p k}(p))(a), \mathbb{P}(F^{n_p k}_k(q_k))(d_k))\\
        &\leq d(\mathbb{P}(F^{n_p k}(p))(a), \mathbb{P}(F^{n_p k}(p))(a_k))\\
                                &+ d(\mathbb{P}(F^{n_p k}_k(p))(a_k), \mathbb{P}(F^{n_p k}_k(q_k))(d_k))\\
				&\leq C_1\; \rho^k + C_3^{n_p k}\, dist(p, q_k)\\
				&\leq C_1\; \rho^k + C\, (C_3^{n_p}C_2^{-n_p}\tau^{\sigma_1})^k.\\
\end{aligned}
\end{equation} 
The estimation in the first term is a consequence of Equation (\ref{frente}). Since, the constant $C_3$ depends only on the functions in Equation (\ref{ph}), we can suppose that $\sigma$ in Equation (\ref{deltak}) was chosen to have the expression on the second term going to zero exponentially fast as $k\to \infty$.

If $\zeta=[z_0,...,z_N]$ is the $su$-path given by Proposition \ref{pivot} for $f$ and $p$, let $$\zeta_1=[z_0,...,z_l]\quad \text{and} \quad \zeta_2=[z_l,...,z_N].$$ Denote $H_{\zeta_i}$ and $h_{\zeta_i}$ the holonomies defined by $\zeta_i$ for $F$ and $\mathbb{P}(F)$ respectively, with $i\in \{1,2\}$. Then, $$H_{\zeta_i}:E^c(p)\to E^c(z_l, f), \quad \text{and} \quad h_{\zeta_i}:\mathbb{P}(E^c(p))\to \mathbb{P}(E^c(z_l, f)),$$ for $i\in \{1,2\}$.

By Proposition \ref{su}, we have $$z_k=z_l^{k}\qquad p_k=z^{k}_N,$$ and $su$-paths for $f_k$, $$\zeta^k_1=[p,...,z_k]\quad \text{and} \quad\zeta^k_2=[p_k,...,z_k].$$ Denote $H_{\zeta^k_i}$ and $h_{\zeta^k_i}$ the holonomies defined by $\zeta^k_i$ for $F_k$ and $\mathbb{P}(F_k)$ respectively, with $i\in \{1,2\}$. Then, $$H_{\zeta_1^k}:E^c(p)\to E^c(z_k, f_k),\quad  \quad H_{\zeta_2^k}:E^c(p_k, f_k)\to E^c(z_k, f_k),$$
$$h_{\zeta_1^k}:\mathbb{P}(E^c(p))\to \mathbb{P}(E^c(z_k, f_k))\quad \text{and} \quad h_{\zeta_2^k}:\mathbb{P}(E^c(p_k, f_k))\to \mathbb{P}(E^c(z_k, f_k)).$$

Although we have two possibilities, $supp\, m_p=\{a\}$ or $supp\, m_p=\{a,b\}$, we are going to consider only the second case, because it imposes more restrictions. Moreover, we can suppose there exist $c, d\in \mathbb{P}(E^c(z_l,f))$ such that $c=h_{\zeta_1}(a)=h_{\zeta_2}(a)$ and $d=h_{\zeta_1}(b)=h_{\zeta_2}(b)$. The other cases are analogous.

In order to simplify the notation, we are going to use the same symbol to denote both a nonzero vector in $E^c_x$ and the corresponding element of $\mathbb{P}(E^c_x)$. 

If $\Phi_k:E^c(z_k,f_k)\to E^c(z_l,f)$ is defined by $\Phi_k=P_{E^c(z_l,f)}\circ \pi_{z_k,{z_l}}$, then for $k$ big enough $\Phi_k$ is an isomorphism. By Corollary \ref{holfinal} and Equation (\ref{nodoscerca}), there exist $C>0$, $\lambda\in (0,1)$ and $K\in \mathbb{N}$ such that for every $k\geq K$ there exists $c_k\in supp\; m^k_{p_k}$ such that 
\begin{enumerate}[label=(\alph*)]
\item $\left\|R_{\beta_k}^{-1}(c)-\Phi_k(H_{\zeta^k_1}(a))\right\|< C\; \lambda^k$,
\item $\left\|R_{\beta_k}^{-1}(d)-\Phi_k(H_{\zeta^k_1}(b))\right\|< C\; \lambda^k$ and 
\item $\left\|c -\Phi_k(H_{\zeta^k_2}(c_k))\right\|< C\; \lambda^k$,
\end{enumerate}
where $R_{\beta_k}:E^c(z_l,f)\to E^c(z_l,f)$ is the rotation of angle $\beta_k>0$ defined by Lemma \ref{pertfinal}.

Since $$\frac{\lambda^k}{\sin^2 \beta_{k}}\rightarrow 0\quad \text{when} \quad k\to \infty,$$ and $\Phi_k$ is an isomorphism, we have that for $k$ big enough, the one-dimensional subspaces generated by $H_{\zeta^k_1}(a)$, $H_{\zeta^k_1}(b)$ and $H_{\zeta^k_2}(c_k)$ are all different. In the projective level this means, $$h_{\zeta^k_1}(a), h_{\zeta^k_1}(b) \neq h_{\zeta^k_2}(c_k).$$ 

On the other hand, since $supp\; m^k_p\subset \{a,b\}$, the invariance by holonomies given by the Invariance Principle implies $$supp\; m^k_{z_k}\subset \{h_{\zeta^k_1}(a), h_{\zeta^k_1}(b) \}.$$ 
Moreover, since $c_k\in supp\; m^k_{p_k}$, then $h_{\zeta^k_2}(c_k)\in supp\; m^k_{z_k}$. 

We arrive to this contradiction because we were assuming that the Invariance Principle could be applied for every $f_k$ with $k\geq k_0$. Then, there exist $k_1\in \mathbb{N}$ such that $\lambda^c_1(f_{k_1})\neq \lambda^c_2(f_{k_1})$ at almost every point. By the definition of $f_{k_1}$, we  know that it is $\epsilon$-close $f$ in the $C^r$ topology. Since $\epsilon>0$ was chosen arbitrarily, this conclude the proof of Theorem B. 
\end{proof}
 
\subsection{Proof of Theorem A} 

\begin{teosim} Let $r\geq 2$, $f\in B^{r}_{\omega}(M)$ and assume the set of periodic points of $f$ is non-empty, then $f$ can be $C^r$-approximated by non-uniformly hyperbolic symplectic diffeomorphisms.
\end{teosim}
The following observations are going to prove that we can reduce the proof to two cases. 

Suppose $M$ is symplectic manifold and $dim\; M=2d$. Given a periodic point $p$ of a symplectic diffeomorphism $f$, define the \textit{principal eigenvalues} to be those $d$ eigenvalues with norm greater that one or with norm equal to one and imaginary part greater than zero and the half of the eigenvalues equals to 1 or -1. If the principal eigenvalues are multiplicative independent over the integers, that is $\prod \lambda_i^{p_i}=1$ with $p_i\in \mathbb{Z}$ implies $p_i=0$ for every $i\in \{1,...,d\}$, we say $p$ is \textit{elementary}.

Let $f\in B^r_{\omega}(M)$ and $p$ be a periodic point with $n_p=per(p)$. By the results in \cite{ROB}, we can suppose that $p$ is elementary. This implies that the eigenvalues of $Df^{n_p}|E^c_p$ satisfy one of the following: 
\begin{enumerate}[label=(\roman*)]
\item there exists $0< \rho < 1$ such that the eigenvalues are $\rho$ and $\rho^{-1}$, or 
\item there exist $x,y\in \mathbb{R}$ such that the eigenvalues are $x+iy$ and $x-iy$ with $x^2+y^2=1$ and they are not a root of unity. 
\end{enumerate}
We are going to call option (i) the hyperbolic case and option (ii) the elliptic case. Then, it is sufficient to prove Theorem A under the hypothesis of $p$ being in one of these cases.
 
\subsubsection{Hyperbolic Case} 

Fix $f\in B^r_{\omega}(M)$ and suppose $p$ is a periodic point satisfying (i). Then, $p$ is a pinching periodic point and we can apply Theorem B to find a symplectic diffeomorphism $g$ $C^r$-arbitrarily close to $f$, with $\lambda^c_1(g)\neq \lambda^c_2(g)$ at almost every point. By the symmetry of the Lyapunov exponents, we have that they are non-zero almost everywhere and therefore we proved Theorem A in this case. 

\subsubsection{Elliptic Case} 

\begin{defn} We say that a periodic point $p$ of period $n_p$ is \emph{quase-elliptic} if there exists $1\leq l \leq d$ such that $Df^{n_p}_p$ has 2l non-real eigenvalues of norm one and its remaining eigenvalues have norm different from one. 
\end{defn} 

We are going to use the following result:

\begin{prop}[Proposition 3.1, \cite{NEW}]\label{ne1} For every $1\leq r\leq \infty$, there exists a residual set $R\subset Diff^{r}_{\omega}(M)$, such that if $f\in R$, then each quasi-elliptic periodic point of $f$ is the limit of transversal homoclinic points.
\end{prop}
Fix $f\in B^r_{\omega}(M)$ and suppose $p$ is a periodic point satisfying (ii). Then, $p$ is a quasi-elliptic periodic point. By Theorem \ref{ne1}, we can suppose that besides of $p$, $f$ has a hyperbolic periodic point $q$. At this point we could apply Theorem B to finish the proof. However, we are going to show that the coexistence of hyperbolic and elliptic periodic points is an obstruction for the rigidity given by Invariance Principle. An argument similar to the one we give here can be found in Remark 2.9 of \cite{ASV}.

Define $F=Df\vert E^c$ and suppose $\lambda^c_1(f)=\lambda^c_2(f)$ almost everywhere. We fix some $\mathbb{P}(F)$-invariant probability measure $m$ with $\pi_{*}m=\mu$ and apply the Invariance Principle. Then, there exits a disintegration $\{m_x: x\in M\}$ such that $\mathbb{P}(F(x))_{*}m_x=m_{f(x)}$ for every $x\in M$.

Since $f$ is accessible, there exist a $su$-path $\zeta$ joining $q$ to $p$. Let $h_{\zeta}$ denote the holonomy defined by $\zeta$ for $\mathbb{P}(F)$. 

Denote $n_p=per(p)$ and $n_q=per(q)$. Let $m_{p}$ and $m_{q}$ be the elements of the disintegration given by the Invariant Principle at $p$ and $q$ respectively. Then, $$\mathbb{P}(F^{n_p}(p))_{*}m_{p}=m_{p},\qquad \mathbb{P}(F^{n_q}(q))_{*}m_{q}=m_{q}\quad \text{and}\quad (h_{\zeta})_{*}m_{q}=m_{p}.$$

Since $q$ is hyperbolic, there exist two points, $a$ and $b$, in $\mathbb{P}(E^c_q)$ such that $supp\, m_q\subset \left\{a, b\right\}$. Therefore, the support of $m_{p}$ contains at most two points. This implies that $\mathbb{P}(F^{n_p}(p))$ has a periodic point of period 1 or 2. However, this contradicts the fact of $p$ being an elliptic periodic point satisfying (ii).

Therefore, we have that the center Lyapunov exponents of $f$ must be different at almost every point. This finish the proof of Theorem A.
\qed

\section{Applications}

In this section we show examples of partially hyperbolic symplectic diffeomorphism that can be $C^r$-approximated by diffeomorphisms in $B^r_{\omega}(M)$ having a periodic point. Then, by Theorem A, we are be able to approximate these examples by non-uniformly hyperbolic systems.

Let $r\geq 2$. $B^{r}_{\omega}(M)$ is the subset of $PH^{r}_{\omega}(M)$ where $f$ is accessible, $\alpha$-pinched and $\alpha$-bunched for some $\alpha>0$, and the center bundle $E^c$ is 2-dimensional. 

For $d\geq 1$, let $\mathbb{T}^{2d}$ denote the $2d$-torus. 
\begin{corol} Let $f:\mathbb{T}^{2d}\longrightarrow \mathbb{T}^{2d}$ be a $C^r$ Anosov symplectic diffeomorphism and $g:\mathbb{T}^{2}\longrightarrow \mathbb{T}^{2}$ a symplectic linear map with eigenvalues of norm one. Then, $f\times g$ can be $C^r$-approximated by non-uniformly hyperbolic diffeomorphisms.
\end{corol}
\begin{proof}
Notice that $f\times g$ is a partially hyperbolic symplectic diffeomorphism with 2-dimensional center bundle. Moreover, $f\times g$ is $\alpha$-pinched and $\alpha$-bunched for some $\alpha>0$ and has a periodic point. 

Theorem A in \cite{SW2} imply that for every $\epsilon>0$, there exists a partially hyperbolic symplectic diffeomorphism $h$ which is accessible and $\epsilon$-close to $f\times g$ in the $C^r$ topology. By the proof in \cite{SW2}, we can suppose that $h$ coincides with $f\times g$ in the orbit of some periodic point and therefore it has itself a periodic point. If $\epsilon$ is small enough, we have $h\in B^r_{\omega}(M)$. Then, we can apply Theorem A to $C^r$-approximate $h$ by non-uniformly hyperbolic diffeomorphisms. This finish the proof.
\end{proof}

Using the same argument than above and Theorem B in \cite{SW2}, we can prove the following result.

\begin{corol} Let $g:\mathbb{T}^{2}\longrightarrow \mathbb{T}^{2}$ be a $C^r$ symplectic diffeomorphism. Then, for every $d\geq 1$ there exists $f:\mathbb{T}^{2d}\longrightarrow \mathbb{T}^{2d}$ a $C^r$ Anosov symplectic diffeomorphism such that $f\times g$ can be $C^r$-approximated by non-uniformly hyperbolic diffeomorphisms.
\end{corol}

Let $\lambda$ be a real parameter. The \textit{standard map} $g_{\lambda}$ of the 2-torus is defined by $$g_{\lambda}(z,w)=(z+w,w+ \lambda \sin (2\pi(z+w))),$$ and it preserves the symplectic form in $\mathbb{T}^2$. By KAM theory, for all values of $\lambda$ near zero, there exists a $C^r$ neighborhood of $g_{\lambda}$ such that any diffeomorphism in this neighborhood has an invariant subset with positive volume where both Lyapunov exponents are zero. The following result shows that if we add some transverse hyperbolicity, we are able to remove the zero Lyapunov exponents. 

\begin{corol} Let $f:\mathbb{T}^{2d}\longrightarrow \mathbb{T}^{2d}$ be a $C^r$ Anosov symplectic diffeomorphism. If $\lambda$ is close enough to zero, $f\times g_{\lambda}$ can be $C^r$-approximated by non-uniformly hyperbolic diffeomorphisms.
\end{corol}
\begin{proof} The argument is the same as before, we need to prove that $f\times g_{\lambda}$ can be $C^r$-approximated by diffeomorphisms in $B^r_{\omega}(M)$ having a periodic point. The only observation we need to make is that $f\times g_{\lambda}$ is $\alpha$-pinched and $\alpha$-bunched for some $\alpha>0$ when $\lambda$ is close enough to zero because $f\times g_0$ is $\alpha$-pinched and $\alpha$-bunched and both conditions are open. The rest of the proof follows using \cite{SW2} and Theorem A.
\end{proof}

\end{document}